\documentclass[11pt,notitlepage]{article}
\usepackage{amssymb,amsthm,bbm,graphicx,placeins,enumitem}
\usepackage{amstext}
\usepackage{natbib}
\usepackage{setspace,mathtools,thmtools,thm-restate}
\usepackage[left=1.25in,right=1.25in,top=1in,bottom=1in,footnotesep=0.8cm]{geometry} 
\usepackage{float}
\usepackage{color,colortbl}
\usepackage{changepage}
\usepackage{booktabs}
\usepackage{array}
\usepackage{multirow}
\usepackage[font=small]{caption}
\usepackage{subcaption}
\usepackage[ruled,norelsize]{algorithm2e}
\usepackage{mathrsfs}  
\usepackage{thm-restate}
\usepackage{csquotes}
\usepackage[noblocks]{authblk}
\usepackage{titlesec}
\usepackage{relsize}

\usepackage{hyperref}
\hypersetup{colorlinks,
linkcolor=blue,
filecolor=blue,
urlcolor=blue,
citecolor=blue}
\allowdisplaybreaks

\DeclareMathOperator*{\argmin}{arg\,min}

\newcommand{\indicate}[1]{\mathbf{1}_{\{#1\}}}

\newtheorem{thm}{Theorem}

\theoremstyle{definition}
\newtheorem{assumption}{Assumption}
\newtheorem{definition}{Definition}

\newtheorem{remark}{Remark}

\DeclareCaptionSubType[Alph]{figure}

\def\Rmax{R_{\text{\textnormal{max}}}}
\def\xmax{x_{\text{\textnormal{max}}}}

\def\E{\mathbf{E}}
\def\P{\mathbf{P}}

\def \rhooperator{\mathlarger{\rho}}

\newcommand\rsmraise[1]{%
  \ifx#1\displaystyle .8\else
    \ifx#1\textstyle .8\else
      \ifx#1\scriptstyle .6\else
        .45%
      \fi
    \fi
  \fi}

\titleformat*{\section}{\Large\bfseries}
\titleformat*{\subsection}{\large      \bfseries}
\titleformat*{\subsubsection}{\normalsize\bfseries}
\titleformat*{\paragraph}{\normalsize\bfseries}
\titleformat*{\subparagraph}{\normalsize\bfseries}

\begin{document}

\title{\textbf{\Large Practicality of Nested Risk Measures for Dynamic\\Electric Vehicle Charging}}
\author{%
Daniel R. Jiang and Warren B. Powell%
}
\maketitle
\vspace{-2em}
\begin{abstract}
We consider the sequential decision problem faced by the manager of an electric vehicle (EV) charging station, who aims to satisfy the charging demand of the customer while minimizing cost. Since the total time needed to charge the EV up to capacity is often \emph{less} than the amount of time that the customer is away, there are opportunities to exploit electricity spot price variations within some reservation window. We formulate the problem as a finite horizon Markov decision process (MDP) and consider a risk-averse objective function by optimizing under a dynamic risk measure constructed using a convex combination of expected value and conditional value at risk (CVaR). It has been recognized that the objective function of a risk-averse MDP lacks a practical interpretation. Therefore, in both academic and industry practice, the dynamic risk measure objective is often not of primary interest; instead, the risk-averse MDP is used as a \emph{computational tool} for solving problems with predefined ``practical'' risk and reward objectives (termed the \emph{base model}). In this paper, we study the extent to which the two sides of this framework are compatible with each other for the EV setting --- roughly speaking, does a ``more risk-averse'' MDP provide lower risk in the practical sense as well? In order to answer such a question, the effect of the degree of dynamic risk-aversion on the optimal MDP policy is analyzed. Based on these results, we also propose a principled approximation approach to finding an instance of the risk-averse MDP whose optimal policy behaves well under the practical objectives of the base model. Our numerical experiments suggest that EV charging stations can be operated at a significantly higher level of profitability if dynamic charging is adopted and a small amount of risk is tolerated.


\end{abstract}


\onehalfspacing
\setlength{\abovedisplayskip}{0.2cm}
\setlength{\belowdisplayskip}{0.2cm}

\section{Introduction}
\label{sec:intro}
The recent popularity of electric vehicles (EVs) has spurred significant research attention in the area of EV storage management from a number of perspectives. One stream of literature, e.g., \cite{Roe2009}, \cite{Clement-Nyns2010}, \cite{Sundstr2010}, \cite{Sortomme2011}, \cite{Rotering2011}, \cite{Gan2013}, \cite{Li2014}, \cite{Wei2014}, takes an aggregate view of EVs and provides possible solutions to the negative impacts of an overloaded electrical infrastructure. Broadly speaking, the literature considers both coordinated and decentralized scheduling approaches to accomplish \emph{peak-shifting}, i.e., flattening the load profile due to EV charging in order to reduce stress on the grid. 

A complementary line of research deals with EV charging policies at the level of individual charging stations, focusing on the optimization of metrics such as charging cost and quality of service (minimal blocking/waiting times in queue-based models). For example, \cite{Zhang2014} integrates a queueing approach with Markov decision process (MDP) theory to optimize mean waiting times under a cost constraint. In \cite{Bayram2011}, \cite{Koutsopoulos2011}, and \cite{Karbasioun2014}, the authors propose optimal operating strategies for the setting where a charging station is paired with storage. The papers \cite{Bashash2010} and \cite{Bashash2011} use a detailed model of a lithium ion battery in order to optimize under the objective of cost of electricity and battery health degradation. The uncertainty in the use of EVs, i.e., the driving pattern of the customer, is taken into account in the charge management policy proposed by \cite{Iversen2014}.

\vspace{7pt} \noindent \textbf{Motivation.} The motivation behind the problem examined in this paper is new, but most closely related to the second stream of literature dealing with individual charging stations. Here, we consider the point of view of a firm that owns charging stations located in public areas, such as parking garages, gas stations, or hotels. The objective is to manage the charging process for a single charger and a single vehicle within the reservation window so that energy costs are low when subjected to stochastic spot prices. When the spot price is high, the manager has an incentive to delay charging until the price falls (at the risk of even higher prices in the future or the inability to fully charge the vehicle).

In order to reap the environmental and societal benefits that come with the widespread adoption of electric vehicles, a dense network of charging stations located nationwide is required, but such infrastructure is not possible without an economic incentive. Unfortunately, there is concern that such an incentive for owning and operating EV infrastructure may not exist in the current climate \citep{Schroeder2012,Chang2012,Robinson2014}. Methods to make ownership of EV charging stations a more viable business model are thus a crucial area of study.

Reservation-based EV charging stations have recently grown in popularity as a way to increase efficiency and decrease range anxiety (i.e., the fear of being stranded away from a charging station with an empty battery). The implementation of reservation-based systems may have also stemmed from social and etiquette considerations: \cite{Caperello2013} suggest that the lack of well-understood behavioral guidelines (e.g., when is it acceptable to unplug another vehicle that is finished charging?) can inhibit the use of public charging stations due to the increased uncertainty. Through surveys, they find that many EV drivers prefer a reservation system so as to mitigate this uncertainty. Since 2011, reservation capabilities at public charging stations have been made available via smart-phone applications (see, e.g., ChargePoint, E-Charge HK in Hong Kong, and NextCharge). Moreover, the rise of the ``sharing economy'' has  led to innovative peer-to-peer EV charging solutions such as Chargie,\footnote{\texttt{https://www.chargie.net/}} where users can make reservations for time on privately owned charging stations.

\vspace{7pt} \noindent \textbf{Problem Overview.} The typical scenario (see Figure \ref{fig:events}) in our model includes the following sequence of events: (1) a customer reserves a charging window $[0,\tau]$ at some time $t < 0$, where $\tau$ can be longer than the time required for a full charge; (2) she drops off the EV at the charging station at $t=0$; (3) the vehicle is (dynamically) charged during the interval $[0,\tau]$; (4) the customer returns at the end of the reservation window, with the expectation that the vehicle is adequately charged (if not, she given an \emph{inconvenience compensation}). From the point of view of the overall model, the reservation length $\tau$ is random variable. Since the dynamic charging policy from $[0,\tau]$ is determined after $\tau$ is known, a family of optimal control problems parameterized by $\tau$ are eventually solved. The theoretical analysis of the control problem, however, is always for a given $\tau = T$.

\begin{figure}[h]
        \centering
         \includegraphics[width=0.8\textwidth]{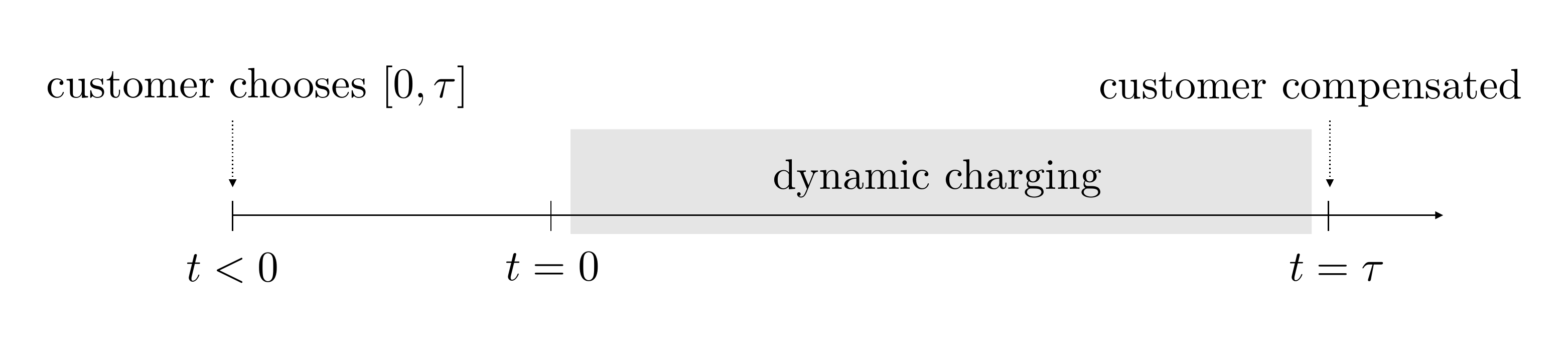}
         \vspace{-10pt}
        \caption{Sequence of Events}
        \label{fig:events}
\end{figure}

In a 2016 patent application by Google describing ``mediator device for smart electric vehicle charging,'' this exact situation is described as follows:
\begin{displayquote}
Modern EVSEs are equipped with technology that allows them to communicate with electric vehicles to control the charging rate. This capability allows the charging station to communicate the amount of charging capacity available to the EV\ldots \, However, in many cases it is not necessary for the EV to charge the battery at the full rate. For example, EV drivers may drive their car[s] to work and leave [them] plugged in all day, even though charging completes in only a few hours. \citep{Wytock2016}
\end{displayquote}
We note that if the EV charging station is run on a first-come, first-served basis without reservations, our model could also be applicable if the smart-phone application asks the customer for an anticipated return time $\tau$. The model is also flexible enough to exist within a larger strategy: for example, if the firm prefers to guarantee a full charge, then the interval $[0,\tau]$ could be split so that dynamic charging occurs first using our model, before switching to a static charging policy at the end.


The decision and risk analysis problem here is to balance the ``risk'' of the customer returning to an undercharged vehicle and the ``reward'' of completing a satisfactory charge at a reasonable cost. The extent to which this is an interesting problem depends largely on the power output of the charging station. For example, residential charging is typically quite slow and there is little opportunity to dynamically adjust the charging rate: Level 1 charging provides around 1.4 kW of output (a few miles per hour of charge) and Level 2 charging provides around 7.7 kW (or 20 miles per hour of charge).


However, technological improvements have resulted in the availability of high-powered chargers and it is now increasingly common for the customer to be away for a length of time that is \emph{longer} than the amount of time needed to fully charge the EV. DC fast charging stations currently provide between 50 kW and 120 kW of power (see, e.g., Tesla's Supercharger), providing 180 to 240 miles per hour of charge. As of 2017, a new tier of charging stations, billed as ``ultra-fast'' DC charging,\footnote{https://www.chargepoint.com/products/commercial/express-plus/}\footnote{https://www.engadget.com/2017/07/16/porsche-installs-first-super-fast-ev-chargers/} are slowly becoming available and can operate at a power rating of 350 to 400 kW. Given the high costs of the DC fast charging technology, estimated at \$65{,}000--\$70{,}000 compared to \$15{,}000-\$18{,}000 for a Level 2 charger \citep{U.S.DepartmentofEnergy2012} (see also \cite{BayAreaCouncil2011} and \cite{Schroeder2012} for similar estimates), it is clear that optimal control policies can help ease the economics of such an investment. Indeed, our conversations with industry colleagues confirm that this problem is faced by owners of EV charging stations today.

\vspace{7pt} \noindent \textbf{The Importance of Risk.} We suppose that the firm can directly interact with a spot energy market. The optimal policy depends on the spot prices of electricity $P_t$ in the following way: if $P_t$ is high, then we have a possible incentive to delay charging and buy later at a lower price. Of course, the trade-off here is that whenever charging is delayed, the risk of an inconvenienced customer (one who returns to a partially charged vehicle) increases. We assume in this paper that an inconvenienced customer is provided a financial compensation for her inconvenience when a vehicle is not charged adequately. The compensation includes a refund plus an additional spot price-dependent amount to cover the cost of purchasing the energy elsewhere. Therefore, the cost function used in our model is of the form
\[
\text{total cost} = \text{cost of energy} + \text{inconvenience compensation},
\]
which we emphasize is a \emph{true financial cost} to the firm. Although a risk-neutral model (i.e., one that optimizes the expectation of the total cost) makes a trade-off between the ``cost of energy'' and the ``inconvenience compensation'' in order to minimize total cost, the following thought-experiment reveals that there are additional preferences that may not be captured under risk-neutrality: given two operating policies $\pi_1$ and $\pi_2$ that achieve the same expected total cost, where $\pi_1$ focuses on a low ``inconvenience cost'' at the expense of a high ``cost of energy'' and $\pi_2$ focuses on the opposite, which policy does one prefer? Since an inconvenienced customer may come with additional externalities (e.g., negative word of mouth, lost future sales) that are difficult to quantify, it is likely that $\pi_1$ is the preferred policy. This hints at the existence of \emph{risk-aversion} toward events where the customer is inconvenienced.


Our proposed model uses the framework of risk-averse MDPs, introduced in \cite{Ruszczynski2010}, to capture the important issue of sequential risk in this problem. In this framework, the cost function remains unchanged: it continues to represent the \emph{true financial cost} to the firm, as discussed above, but we replace the expected value operator with a \emph{dynamic risk measure} \citep{Frittelli2004,Riedel2004,Pflug2005,Boda2006,Cheridito2006,Acciaio2011} in order to account for uncertain outcomes in a different way. When a \emph{time-consistency} property holds for the dynamic risk measure (see \cite{Cheridito2006} and \cite{Ruszczynski2010}), the risk-averse objective function for a finite horizon $T$ becomes: 
\begin{equation}
\min_{\pi \in \Pi} \; C_0^\pi + \rhooperator_0 \bigl ( C_1^\pi + \rhooperator_1 \bigl( C_2^\pi + \cdots + \rhooperator_{T-1}( C_T^\pi ) \cdots \bigr) \bigr),
\label{eq:obj1}
\end{equation}
where $\Pi$ is a set of policies, $\{C_t^\pi\}$ are costs incurred by following policy $\pi$, and $\{\rhooperator_t\}$ are one-step coherent risk measures (i.e., components of the overall dynamic risk measure). A major computational advantage of this representation is that the familiar Bellman recursion applies for the value functions of the optimal risk-averse policy, allowing for established solution techniques to be adapted to the risk-averse setting. We provide a theoretical analysis of the \emph{impact of risk-aversion} on the structure of the optimal policy for our EV charging problem. Related work to ours include risk-averse single and multi-stage newsvendor models \citep{Bouakiz1992,Eeckhoudt1995,Martinez-De-Albeniz2006,Gotoh2007,Ahmed2007,Chen2009,Choi2011} and risk-averse joint-inventory and pricing models \citep{Chen2007,Yang2016}. For single stage problems, the risk metrics used in these papers include conditional value at risk (CVaR) \citep{Rockafellar2000}, exponential utility, and mean-variance, while for multi-stage problems, they employ either additive utility functions or dynamic (coherent) risk measures. \cite{Choi2011} discuss the merits of \emph{translation invariance} and \emph{positive homogeneity}, two properties that our choice of methodology, dynamic coherent risk measures, possesses which the utility function framework does not.

\vspace{7pt} \noindent \textbf{Interpretability of the Risk-Averse Objective.} In this paper, we address the issue of \emph{interpretability} of (\ref{eq:obj1}) in a new way. It has been acknowledged that the main obstacle to the practicality of dynamic risk measures in MDPs is the lack of a clear economic interpretation \citep{Rudloff2014,Iancu2015}. If $\rhooperator_t = \mathbf{E}_t$ for all $t$, then (\ref{eq:obj1}) is equivalent to a quantity familiar to any practitioner, the expected total cost: $\min_\pi \mathbf{E} \bigl[ \sum_t C_t^\pi \bigr]$. On the other hand, if $\rhooperator_t$ is chosen to be, say, CVaR, then (\ref{eq:obj1}) is evaluated with a nested version of CVaR \citep{Cheridito2009a} and does not further simplify into a practical quantity. Nevertheless, the theoretical grounding and associated computational advantages of the framework \citep{Ruszczynski2010,Philpott2012,Philpott2013} makes it an attractive model for many: in fact, the Brazilian National System Operator (ONS) has \emph{operationalized} a monthly hydrothermal planning strategy based on an MDP under nested mean-CVaR \citep{Maceira2015}. To account for the interpretability issue, it is very common in both the literature and in practice to evaluate risk-averse policies generated by a risk-averse MDP along two dimensions, \emph{practical risk} and \emph{practical reward}, which may or may not be directly related to the cost function of the MDP. See Figure \ref{fig:practical} for an illustration of this approach. A few illustrative examples are listed below:
\begin{itemize}
\item \cite{Philpott2012} solve a hydro-thermal scheduling problem under a nested mean-CVaR risk measure. The practical reward is taken to be the (negative) ``total operational cost'' and the practical risk is the ``number of violations where the national storage level is below a minimum threshold.'' The authors explicitly acknowledge the discrepancy by stating that the dynamic risk measure objective ``is arguably different from focusing on the distribution of the total annual cost, and controlling the extent of its upper tail.''
\item \cite{Shapiro2013} optimize the Brazilian interconnected power system under a nested mean-CVaR risk measure, but effectiveness of the policy is gauged using the ``expected cost'' versus the ``90\%, 95\%, and 99\% quantiles of the cost.''
\item \cite{Cavus2014a} compute risk-averse policies under a nested mean-semideviation risk measure for a credit limit decision problem and evaluates them by observing the profit distribution, where the implicit trade-off is between ``expected profit'' and ``frequency of negative profits.''
\item \cite{Lin2015} solve the \cite{Almgren2001} trade execution model under a general dynamic measure of risk. The numerical work evaluates the policies on the practical metrics of (negative) ``mean cost'' versus the ``mean cost of the 30\% worst samples.''
\item \cite{Maceira2015} use a nested mean-CVaR for the hydro-planning MDP, but chooses the practical reward to be (negative) ``expected value of operational cost'' and practical risk to be ``expected value of energy not supplied.''
\item \cite{Collado2017} consider a stochastic path detection problem under nested mean-semideviation. The practical trade-off being considered is the ``expected non-detection probability'' versus the ``variance of the non-detection probability.''
	\end{itemize}
We emphasize that in each of these papers, the approach of Figure \ref{fig:practical} is used. Although it is related, the dynamic risk measure objective used to find the policy does not match the practical objectives used for evaluation. Thus, we can view the risk-averse MDP as a \emph{computational tool} for generating ``risk-averse'' policies that solve another model, which we term the \emph{base model}. Based on trade-offs studied in the papers listed above, it is reasonable to specify a base model of the form 
\begin{equation}
\displaystyle \text{maximize}_{\pi} \; \; [\text{Practical Reward}](\pi) \; \; \text{subject to} \; \; [\text{Practical Risk}](\pi) \le \varepsilon,
\end{equation}
where $\pi$ is an operating policy and $\varepsilon$ is a parameter that controls the risk. Hence, the approach under discussion is when $\pi$ is chosen to be the solution of a risk-averse MDP.

\begin{figure}[h]
        \centering
                \includegraphics[width=0.85\textwidth]{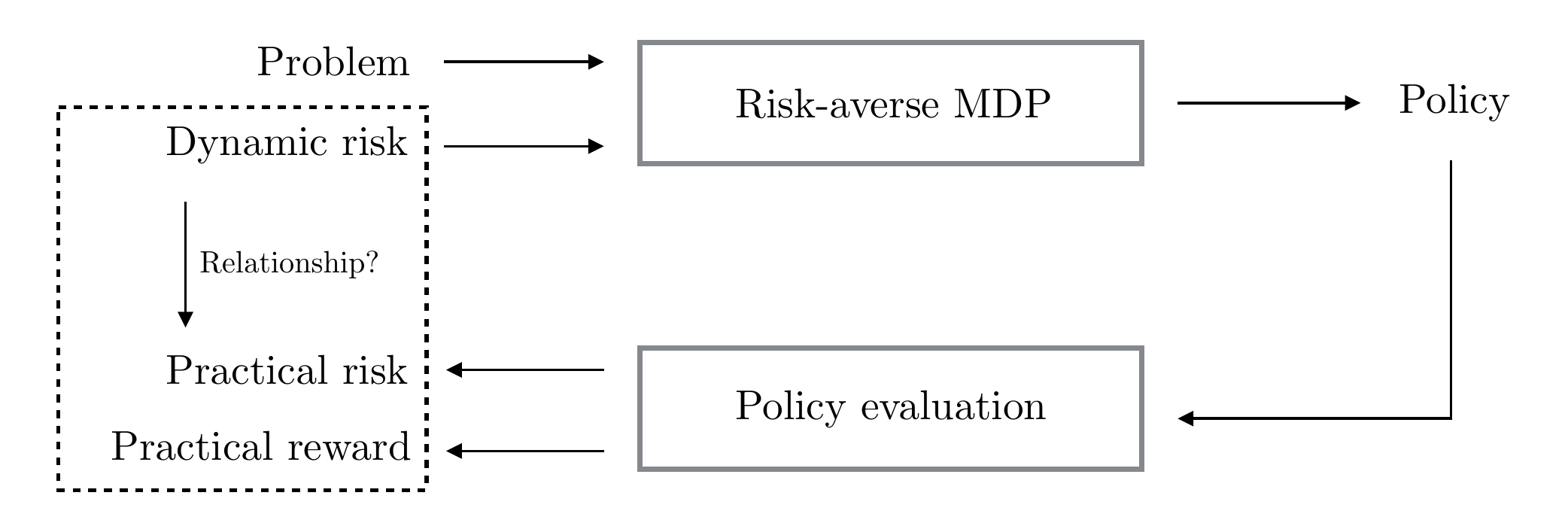}

        \caption{Evaluation using Practical Metrics of Risk and Reward}
        \label{fig:practical}
\end{figure}
We note that this approach is typically not applied in a rigorous manner; the relationship between risk-aversion in the dynamic risk measure sense and the practical risk and reward is generally not well-understood. Our paper poses the question: does increased risk-aversion in the risk-averse MDP lead to lower practical risk in the base model as well? Such a property is desired as it indicates a \emph{compatibility} between the risk-averse MDP and the practical objective. Without such a compatibility, the motivation for utilizing the risk-averse MDP model is quite weak. In this paper, we consider this question in the context of the dynamic EV charging problem along with several practical metrics that are of interest to industry. We show conditions under which compatibility holds for our model.

In current practice, the decision-maker is often limited by a computing budget and thus heuristically selects a few MDPs to solve. If one of the resulting policies generates an acceptable trade-off between practical risk and reward, then it is chosen for implementation; see, e.g.,\cite[Tables 6-7]{Philpott2012}, \cite[Figure 11]{Shapiro2013}, \cite[Table 8]{Cavus2014a}, and \cite[Figure 6]{Maceira2015} for examples of this general ad-hoc approach. Ideally, the decision-maker should investigate a larger set of risk-averse MDP policies. Along these lines, we offer a method of approximating the practical risk and reward for a spectrum of risk-averse policies using \emph{polynomial optimization and sum of squares constraints} \citep{Ahmadi2015}. The idea is that by solving a small number of risk-averse MDPs, our structural results allow us to reliably approximate the performance of an entire set of risk-averse optimal policies, providing practitioners an accessible means to explore a wider range of risk preferences. We then take advantage of the approximations to provide a systematic method for selecting a policy corresponding to a risk-averse MDP by analyzing the trade-offs of the practical metrics.


\vspace{7pt} \noindent  \textbf{Our Results.} The main contributions of this paper, some of which are of general interest beyond the EV application, are summarized below.
\begin{itemize}
\item We present a novel risk-averse, sequential model of the EV charging problem. The model is formulated using the framework of risk-averse MDPs with dynamic risk measures, constructed using one step risk measures that are a convex combination of expectation and conditional value at risk.
\item We analyze the structural properties of the optimal policy. A main result here states that the charging thresholds of the optimal policy are nonincreasing in the current spot price (under a spot-price model incorporating seasonality, mean-reversion, and jumps) for any degree of dynamic risk-aversion.
\item Related to the above results, we find conditions under which our risk-averse MDP is \emph{compatible} (a notion that we will define) with a class of practical metrics of risk and reward when the approach of Figure \ref{fig:practical} is adopted. This can be taken as justification that the use dynamic risk measures in risk-averse MDPs is an appropriate methodology in the case of the dynamic EV charging problem we present here.
\item We propose an approximation strategy based on regression and polynomial optimization in order to guide the policy selection process, i.e., we address the question: how do we choose a risk-averse MDP such that the resulting optimal policy behaves well under the practical metrics?
\item Finally, we demonstrate the proposed methods on a case study using spot price data from the California ISO (CAISO) and parking time distributions estimated from public garages in Santa Monica, California. 
\end{itemize}

The paper is organized as follows. First, we present in Section \ref{sec:mathematicalformulation} the mathematical model of the problem. In Section \ref{sec:structure}, we provide a structural analysis of the model. Next, in Section \ref{sec:compat}, we analyze 
the relationship of the policy to practical metrics of risk and reward and identify when our notion of compatibility holds. Lastly, the case study and the associated numerical work is presented in Section \ref{sec:num}.







\section{Mathematical Formulation}
\label{sec:mathematicalformulation}
In this section, we first give a mathematical formulation of the dynamic EV charging problem. We then proceed to give a brief overview of dynamic risk measures in MDPs, along with the objective function used in the base model and objective used in the risk-averse MDP.

\subsection{Dynamic EV Charging Model}
\label{sec:chargingmodel}
Each realization $T$ of the reservation length $\tau$ and each risk preference $\beta$ defines a sequential decision problem of length $T+1$, and in this paper, we are interested in the entire class of problems parameterized by $T$ and $\beta$ (hence, many relevant quantities will be indexed by both $T$ and $\beta$). In this section and Section \ref{sec:structure}, we define and analyze the structure of the MDP under a fixed $T$ and $\beta$. Next, in Section \ref{sec:compat}, we characterize the relationship between MDPs with varying $\beta$ and also seek to understand the how the optimal policies behave when evaluated using practical metrics of risk. The full distribution of $\tau$ is then taken into account in the numerical experiments of Section \ref{sec:num}.

In our MDP, charging decisions are made at $t \in \{0,1,\ldots, T-1\}$, which are measured in units of 15 minutes. The time-based access fee to the charging station, paid by the customer per 15 minutes, is denoted $c_f$. The customer returns at time $T$ and if applicable, receives an inconvenience compensation at time $T+1$ if an inadequate charge was provided (see Section \ref{subsec:compensation}). At every period $t$, our charging decision is a rate $x \in [0,\xmax]$ and using normalized units, we assume that $x$ units of energy are obtained between $t$ and $t+1$. Let $R_t \in [0,\Rmax] =: \mathcal R$ be the amount of energy in storage at time $t$, where the \emph{transition function} for $R_t$ is given by $R_{t+1} = R_t + x_t$, provided that the charging decision $x_t$ is chosen in the constraint set $\mathcal X(R_t) = \{x:0 \le x \le \min\{ \Rmax - R_t, \, \xmax \}\}$. 

Given our previous discussion on the wide range of charging station power ratings (1.4 kW to 400 kW), we note that it is increasingly likely that the maximum charge per period $\xmax$ is larger than $\Rmax$. This paper will explicitly consider two parameter regimes: (1) the \emph{general charging regime} where $\xmax$ is unrestricted and (2) the \emph{fast charging regime} where $\xmax$ is larger than $\Rmax$. We note that $\xmax$ depends on both the power output of the charging station and the length of a decision epoch; thus, the ``fast charging regime'' is a slight misnomer in that it arises when either the power output is high \emph{or} the length of a decision epoch is large. The reason for the distinction between the two parameter regimes is that certain theorems (the results of Section \ref{sec:compat}) only hold in the fast charging case. Unless specifically noted, we consider the general charging regime, where no assumptions are made on $\xmax$.

\subsubsection{Spot Price Dynamics} Now let us specify the dynamics of the spot price process $\{P_t\}$. Electricity prices exhibit several characteristics that should be captured: seasonality, mean-reversion, jumps, and potentially negative prices \citep{Eydeland2003,Kim2011,Zhou2015}. Because many conversations regarding dynamic EV charging technology are occurring in California, we obtained electricity price data from the California (CAISO) fifteen-minute market. These data, as shown in Figure \ref{fig:caiso}, exhibits frequent negative prices but relatively light tails when compared to other markets (see, e.g., the five-minute NYISO real-time market). 

\begin{figure}[h]
        \centering
                \includegraphics[width=0.8\textwidth]{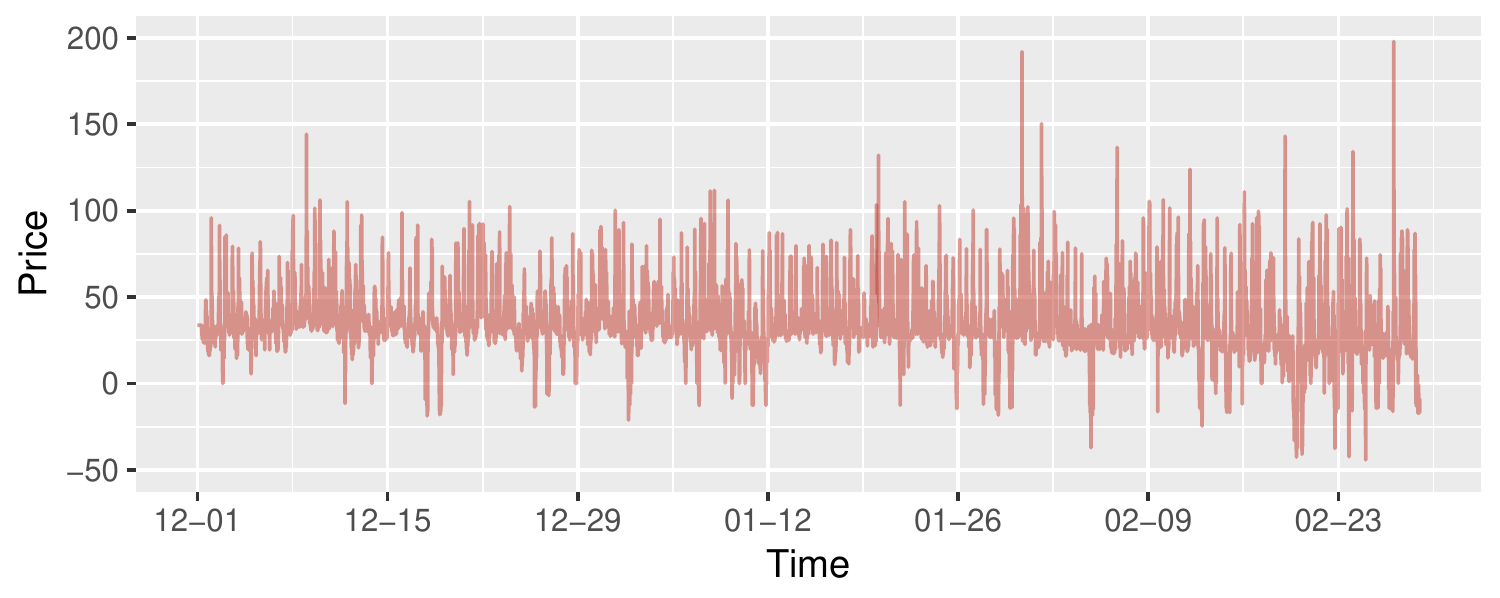}

        \caption{Historical Fifteen-Minute CAISO Prices, Winter 2016-2017}
        \label{fig:caiso}
\end{figure}

Given these observations and the discrete-time decision problem, a reasonable choice to accommodate for the negative prices is to employ a discrete-time version of the jump-diffusion model of \cite{Cartea2005} to directly fit the spot prices, rather than the log spot price. The continuous-time model is given by $P^\textnormal{c}_t = g(t) + Y^\textnormal{c}_t$. Here, $g(t)$ is a deterministic (sinusoidal) seasonality function and $Y_t^\textnormal{c}$ consists of a mean reverting component and a jump component; the dynamics are given by $dY^\textnormal{c}_t = \kappa_Y \, (\mu_Y - Y^\textnormal{c}_t) \, dt + \sigma_Y \, dW_t + J dN_t$, where $\kappa_Y >0$ is the mean reversion parameter, $\theta_Y$ is the long term mean of the process, $\sigma_Y > 0$ is the volatility, $W_t$ is a standard Brownian motion, $J$ is a normally distributed jump size with mean $\mu_J$ and variance $\sigma_J^2$, and $N_t$ is a Poisson process with rate $\lambda$ (all independent from one another). 

The discrete-time approximation $P_t$ of the continuous-time model at resolution $\Delta t$ can be obtained by letting 
\begin{equation}
P_t = g(t) + Y_t \quad \text{and} \quad Y_{t+\Delta t} = Y_t \, e^{-\kappa_Y  \Delta t} + \mu_Y \, (1- e^{-\kappa_Y \Delta t}) + \xi_{t+\Delta t} + X_{t+\Delta t} \, J_{t+\Delta t},
\label{eq:discretetimeou}
\end{equation}
where $\{\xi_t\}$ is a sequence of iid $\mathcal N \bigl(0,\sigma_Y^2/(2\kappa_Y) (1-e^{-2\kappa_Y \Delta t}) \bigr)$ random variables, $\{X_t\}$ is a Bernoulli process with parameter $\lambda \, \Delta t$, and $\{J_t\}$ is a sequence of iid $\mathcal N(\mu_J, \sigma_J^2)$ jumps (a Bernoulli process approximation of the Poisson process, which assumes at most one jump per period). $\xi_{t+\Delta t}$, $X_{t+\Delta t}$, and $J_{t+\Delta t}$ are mutually independent. We utilize (\ref{eq:discretetimeou}) with $\Delta t = 1$ for the remainder of this paper and assume that each time step $\Delta t$ is 15 minutes. Let $(\Omega, \mathcal F, \mathbf{P})$ be the probability space on which all random variables are defined, where $\{\xi_t\}$, $\{X_t\}$, and $\{J_t\}$ are adapted to a filtration $\{\mathcal F_t\}_{t=0}^T$, with $\{\varnothing,\Omega\} = \mathcal F_0 \subseteq \mathcal F_1 \subseteq \cdots \subseteq \mathcal F_T \subseteq \mathcal F$.

Note that by (\ref{eq:discretetimeou}), we can define a random variable
\begin{equation}
\psi_{t+1} = g(t+1)-g(t) \, e^{-\kappa_Y} + \mu_Y \, (1- e^{-\kappa_Y}) + \xi_{t+1} + X_{t+1} \, J_{t+1},
\label{eq:psidef}
\end{equation}
which does not depend on the current spot price $p$, such that $p \, e^{-\kappa_Y} + \psi_{t+1}$ has the same distribution as $(P_{t+1} \, | \, P_t = p)$. Similarly, we define $\psi_{t+1,Y} = \psi_{t+1} - g(t+1)$, also independent of the current spot price $p$, such that $p \, e^{-\kappa_Y} + \psi_{t+1,Y}$ has the same distribution as $(Y_{t+1} \, | \, P_t = p)$. Both $\psi_{t+1}$ and $\psi_{t+1,Y}$ are distributed as a mixture of normals. Throughout this paper, we use the notation $P_{t+1}(p)$ to represent a random variable equal in distribution $(P_{t+1} \, | \, P_t=p)$ and $Y_{t+1}(p)$ to represent $(Y_{t+1} \, | \, Y_t = p-g(t))$ in order to clarify certain expressions. Calibration details and the resulting model parameters are discussed in Section \ref{sec:num}.

\subsubsection{Customer Compensation} 
\label{subsec:compensation}
To properly assess the trade-offs in our model, we need to specify the compensation structure for a customer who returns to a vehicle with charge level $R_T$. Because our current optimal control problem has only seen preliminary discussion \citep{Wytock2016}, there is no concrete, implemented example of how (and if) customers should be compensated in the case of a partial charge. Here, we present one reasonable choice that has arisen in discussion and remark that alternative structures satisfying similar convexity and monotonicity properties can be introduced into the model without significant difficulty.

Under a standard charging policy, the customer expects her vehicle to either be fully charged up to $\Rmax$ or has received a continuous charge for $T$ periods, i.e., charged up to $R_0 + T  \xmax$. Therefore, the energy shortage is $h(R_T) = [\min(R_0 + T  \xmax,\Rmax) - R_T]$. Let $p_\text{ref}$ be a deterministic amount that represents the retail price per unit of energy, so that the customer is entitled to a refund of at least $h(R_T) \, p_\text{ref}$. The compensation also includes two additional payments, an \emph{inconvenience amount} of $\gamma_h$ per unit of shortage and a \emph{market-dependent amount} of $\gamma_Y(P_{T+1}-g(T+1))$ per unit of shortage for some \emph{positive, increasing} function $\gamma_Y$ (see Assumption \ref{ass:compensationlip}). The idea here is that if the shortage occurs in a time of high market prices (i.e., relative to the seasonality function), then the customer receives an additional payment to cover her costs at period $T+1$. The total inconvenience compensation at time $T+1$ is given by:  
\[
c_{T+1}(S_T, P_{T+1}) = [1+\gamma_h\, h(R_T) + \gamma_Y(P_{T+1}-g(T+1))] \, h(R_T) \, p_\text{ref}.
\]
The positivity of $\gamma_Y$ is motivated by the notion that we should never reduce compensation payments to the customer, even in times of low or negative prices.

\subsubsection{One-Period Cost Function}

The \emph{state variable} for this problem is $S_t = (R_t,P_t)$, which takes values in the state space $\mathcal S = \mathcal R \times \mathbb R$.  We are now ready to define the cost function, which is given by $c_t(S_t, x_t) = (x_t P_t - c_f) \, \indicate{t < T}$ for $t=0, 1, \ldots, T$ and the terminal cost is $c_{T+1}(S_T, P_{T+1})$, as defined above. We emphasize that both the cost of energy (for $t < T$) and the inconvenience compensation (for $t=T+1$) should be viewed as \emph{actual financial costs} to the firm, and should not be interpreted to be artificial penalties intended to generate risk-averse policies. All risk preferences are added after the model of the problem is defined.

Let $\{x_0^\pi, x_1^\pi,\ldots,x_{T-1}^\pi\}$ be a charging policy (indexed by $\pi$) where $x_t^\pi: \mathcal S \rightarrow [0,\xmax]$ is the decision function at time $t$. When following a policy $\pi$, we produce a sequence of states $S_0^\pi, S_1^\pi, \ldots, S_{T}^\pi$ and obtain the costs $C_t^\pi = c_{t}(S_{t}^\pi, x_{t}^\pi(S_{t}^\pi))$, for $t=0,1,\ldots, T$ and $C_{T+1}^\pi = c_{T+1}(S_T^\pi,P_{T+1})$ (note that $C_T^\pi = 0$). Let $R_t^\pi$ be the level of charge at time $t$ if a policy $\pi$ is followed. Let $\mathbf{S}^\pi_T = (S^\pi_0, S^\pi_{1}, \ldots, S^\pi_{T+1})$ be the vector representing the sequence of states visited under policy $\pi$ with $S_t^\pi = (R_t^\pi, P_t)$.

\subsection{The Base Model \& Practical Metrics}
\label{sec:prac}
 We now define a broad class of \emph{practical reward} and \emph{practical risk} metrics, along with a few motivating examples. Let $f^+_T$ be a function mapping a state vector for the problem of horizon $T$ into a ``reward outcome,'' so that the quantity $f^+_T(\mathbf{S}^\pi_T)$ is the random variable describing the practical reward of policy $\pi$. Similarly, define a function $f^-_T$ so that $f^-_T(\mathbf{S}^\pi_T)$ is the practical risk of policy $\pi$. 
Finally, let $r_\text{base}$ and $\rho_\text{base}$ be monotone mappings from random variables into $\mathbb R$ (e.g., $\mathbf{E}$ or $\text{VaR}_{0.9}$) so that our metrics are given by
\begin{equation*}
\textnormal{Practical Reward} = r_\text{base}\,[f^+_\tau(\mathbf{S}^\pi_\tau)] \quad \text{and} \quad \textnormal{Practical Risk} = \rho_\text{base}\,[f^-_\tau(\mathbf{S}^\pi_\tau)].
\end{equation*}
Note that we have replaced $T$ with the random variable $\tau$, in order to incorporate the uncertainty of the reservation horizon.

Although we have given a general definition above so that this framework can be applied in other application settings, in the EV charging problem, our primary goal is to use dynamic charging to reduce the cost of energy. Thus, the most sensible form of the practical reward is given by $r_\text{base} = \mathbf{E}$ and
\begin{equation}
 f_\tau^+(\mathbf{S}^\pi_\tau) = c_f \, \tau - \sum_{t=1}^\tau (R_t^\pi - R_{t-1}^\pi) \, P_t - [1+\gamma_h\, h(R^\pi_\tau) + \gamma_Y(Y_{\tau+1})] \, h(R^\pi_\tau) \, p_\text{ref},
\label{eq:practicalreward}
\end{equation}
 so that $r_\text{base}\,[f^+_\tau(\mathbf{S}^\pi_\tau)]$ represents the negative of the expected cost. On the other hand, any metric that we might reasonably want to minimize could be used for the practical risk. A few examples are given below.
\begin{itemize}
\item[--] $\rho_\text{base} = \mathbf{E}$ and $f_\tau^-(\mathbf{S}^\pi_\tau) = [1+\gamma_h\, h(R^\pi_\tau) + \gamma_Y(Y_{\tau+1})] \, h(R^\pi_\tau) \, p_\text{ref}$ gives the expected inconvenience compensation.
\item[--] Similarly, one could also consider letting $\rho_\text{base} = \text{VaR}_{0.9}$ and keeping $f_\tau^-$ as defined above so that the practical risk is the $90^{\text{th}}$ quantile of the inconvenience compensation.
\item[--] $\rho_\text{base} = \mathbf{E}$ and $f_\tau^-(\mathbf{S}^\pi_\tau) = \mathbf{1}_{\{R_\tau / \Rmax \, \le \, (1-\delta)\}}$ represents the probability that $R_\tau$ is less than some threshold of being full.
\item[--] $\rho_\text{base} =\text{CVaR}_{0.90}$ and $f_\tau^-(\mathbf{S}^\pi_\tau) = h(R^\pi_\tau)$ represents the average amount of shortage in the 10\% worst cases.
\end{itemize}
Note that each of these risk metrics have a direct economic interpretation that can be useful in practice. We define the \emph{base model} to be
\begin{equation}
\displaystyle \text{maximize}_\pi \; \; r_\text{base}[f^+_\tau(\mathbf{S}^\pi_\tau)] \; \; \text{subject to} \; \; \rho_\text{base}[f^-_\tau(\mathbf{S}^\pi_\tau)] \le \varepsilon,
\label{eq:tradeoffproblem}
\end{equation}
for some ``trade-off'' parameter $\varepsilon$. Constrained MDPs, such as (\ref{eq:tradeoffproblem}), are known to be very difficult to solve in general (see \cite{Altman1999} for an overview), so many authors resort to using a risk-averse MDP to approximate the base model. To reiterate our previous discussions, we are interested in examining a notion of compatibility when policy evaluation is performed using (\ref{eq:tradeoffproblem}) and policy optimization is done using a risk-averse MDP. 

\subsection{The Risk-Averse MDP Model}
\label{sec:dynamicmdp}
In this section, we discuss the risk-averse MDP model of \cite{Ruszczynski2010}, which serves as a surrogate to the computationally intractable base model given in (\ref{eq:tradeoffproblem}). The eventual goal of this paper is to search over a parameterized class of MDPs in order to obtain an acceptable solution to (\ref{eq:tradeoffproblem}).

The objective function of a risk-averse MDP is specified using a time-consistent dynamic risk measure $\rhooperator_{0,T}$ of the \emph{sequence} of downstream costs $C^\pi_0, C^\pi_1, \ldots, C^\pi_{T+1}$ generated by a policy $\pi$. The objective is written using the notation $\rhooperator_{0,T}(C^\pi_0,C^\pi_1, \ldots, C^\pi_{T+1})$; in the risk-neutral case, we would have $\rhooperator_{0,T}(C^\pi_0,C^\pi_1, \ldots, C^\pi_{T+1}) = \mathbf{E}\Bigl[\sum_{t=0}^{T+1} C^\pi_t  \Bigr]$, but in general, we cannot assume such an additive form. The notion of time-consistency can be summarized as: if we prefer one stream of costs at some time in the future, and the costs from now until then are identical, then we must also prefer the same stream of costs \emph{today}; see \cite{Ruszczynski2010} for a precise definition. Let $\mathcal Z_t$ be the space of $\mathcal F_t$ measurable random variables. It can be shown, under some mild assumptions, that time-consistency results in the nested formulation
\begin{align}
\rhooperator_{0,T}(C^\pi_0,&C^\pi_1,\ldots, C^\pi_{T+1}) = C^\pi_0+ \rhooperator_0(C^\pi_1 + \rhooperator_1(C^\pi_2+ \cdots + \rhooperator_{T}(C^\pi_{T+1})\cdots)),
\label{eq:nestedrho}
\end{align}
for some \emph{one-step conditional risk measures} $\rhooperator_t : \mathcal Z_{t+1} \rightarrow \mathcal Z_t$.


The traditional risk-neutral optimization problem (see, e.g., \cite{Puterman}) is
\begin{equation}
\min_{\pi \in \Pi} \;  \mathbf{E}\left[\sum_{t=0}^{T+1} C^\pi_t  \right] = \min_{\pi \in \Pi} \;  C_0^\pi + \mathbf{E}_0\bigl( C_1^\pi + \cdots + \mathbf{E}_{T}( C_{T+1}^\pi ) \cdots \bigr),
\label{eq:rnobjective}
\end{equation}
where $\mathbf{E}_t$ is shorthand for $\mathbf{E}( \, \cdot \, | \, \mathcal F_t)$. We define the risk-averse optimization problem as $\min_{\pi \in \Pi} \;  \rhooperator_{0,T}(C_0^\pi, C_1^\pi,C_2^\pi, \ldots, C_{T+1}^\pi)$, where the expected sum of costs is replaced with a dynamic risk measure $\rhooperator_{0,T}$ acting on the sequence of costs.
Applying the nested formulation of (\ref{eq:nestedrho}), we see that the risk-averse formulation is a clear analog of the risk-neutral case (\ref{eq:rnobjective}) with the conditional expectations replaced with one-step risk conditional risk measures:
\begin{equation}
\min_{\pi \in \Pi} \; C_0^\pi+  \rhooperator_0\bigl( C_1^\pi + \rhooperator_1(C_2^\pi+ \cdots + \rhooperator_{T-1}( C_{T}^\pi ) \cdots )\bigr).
\label{eq:rmdp}
\end{equation}
In this paper, we follow several successful applications of risk-averse MDPs in the literature \citep{Philpott2012,Shapiro2013,Maceira2015} and take $\rhooperator_t$ to be the mean-CVaR risk measure. Every one-step conditional risk measure $\rho_t:\mathcal Z_{t+1} \rightarrow \mathcal Z_{t}$ has a ``static'' (as opposed to ``dynamic'') counterpart that maps a random variable to $\mathbb R$, instead of another random variable. To avoid confusion, we skip directly to the static case.

Define the risk measures $\textnormal{CVaR}_{\alpha_t} : \mathcal Z \rightarrow \mathbb R$ \citep{Rockafellar2000} and $\rhooperator_{\beta_t} :\mathcal Z \rightarrow \mathbb R$, where $\mathcal Z$ is a space of random variables, to be
\begin{equation}
\begin{aligned}
\textnormal{CVaR}_{\alpha_t}(X) &= \inf_u \bigl\{  u + (1-\alpha_t)^{-1} \, \mathbf{E} \bigl[ (X-u)^+  \bigr] \bigr\},\\
\rhooperator_{\beta_t}(X) &= (1-\lambda_t) \, \mathbf{E}(X) + \lambda_t \,\textnormal{CVaR}_{\alpha_t}(X),
\label{eq:cvarstaticdefn}
\end{aligned}
\end{equation}
where $X$ is a random cost in $\mathcal Z$, $\lambda_t \in [0,1]$ controls the emphasis on the tail risk, $\alpha_t \in (0,1)$ controls the length of the tail, and $\beta_t = (\lambda_t, \alpha_t)$ is the \emph{degree of risk-aversion} at time $t$. Note that $\rhooperator_{\beta_t}$ is a \emph{coherent risk measure}, as axiomatized in \cite{Artzner1999}, meaning that it satisfies the properties of convexity, monotonicity, translation invariance, and positive homogeneity. The \emph{value at risk} for some risk level $\alpha_t$ of a random variable $X$ is defined to be the $\alpha_t$-quantile, or $\textnormal{VaR}_{\alpha_t} (X) = \inf  \{ u: \mathbf{P}(X \le u) > \alpha_t \}$.


Thus, we henceforth consider the optimization problem (\ref{eq:rmdp}) with $\rhooperator_t:\mathcal Z_{t+1} \rightarrow \mathcal Z_t$ taken to be the conditional counterpart to $\rho_{\beta_t}:\mathcal Z \rightarrow \mathbb R$; see \cite{Ruszczynski2006,Ruszczynski2010} for details and examples. Let $\beta = (\lambda_0, \alpha_0, \lambda_1, \alpha_1, \ldots, \lambda_{T}, \alpha_{T})$ control the \emph{degree of dynamic risk-aversion} or \emph{risk preference} of the MDP. The associated dynamic risk measure objective function is specified using $\beta$ is written $\rhooperator_{0,T}(C_0^\pi, C_1^\pi, \ldots, C_{T+1}^\pi \, | \, \beta)$.

Let $\pi_T(\beta)$ be the (index of) an optimal policy for the decision problem of horizon $T+1$ where $\rhooperator_t$ are specified using $\beta$. For convenience, we denote the decision function at time $t$ for the optimal policy by $x_{t,T}(\, \cdot \, | \, \beta)$. The theorem below, proved in \cite{Ruszczynski2010}, serves as the foundation for risk-averse MDPs; it is an analog of the well-known risk-neutral Bellman recursion, with expectations replaced by one-step risk measures (in our case, they are taken to be $\rhooperator_{\beta_t}$). 

\begin{thm}[Risk-Averse Bellman Recursion, \cite{Ruszczynski2010}]
\hfill \\
Let the boundary condition be $V_{T,T}(s \, | \,\beta) = \rhooperator_{\beta_T}[c_T(s,p \, e^{-\kappa_Y} + \psi_{T+1,Y})]$ for any state $s \in \mathcal S$. For each of the earlier time periods $t  \in \{0,1,\ldots,T-1\}$, let
\begin{align*}
V_{t,T}(s \, | \, \beta) &=\min_{x \in \mathcal X(r)} c_t(s, x)+ \rhooperator_{\beta_t} \bigl[ V_{t+1,T}(r+x, p \, e^{-\kappa_Y} + \psi_{t+1} \, | \, \beta) \bigr].
\end{align*}
The optimal policy to (\ref{eq:rmdp}) is given by
\[
x_{t,T}(s \, | \, \beta)\in \argmin_{x\in \mathcal X(r)}  \; c_t(s, x) + \rhooperator_{\beta_t} \bigl[ V_{t+1,T}( r+x, p \, e^{-\kappa_Y} + \psi_{t+1} \, | \, \beta) \bigr],
\]
\label{thm:bellman}
i.e., it is greedy with respect to the optimal value function.
\end{thm}

We now introduce some useful notation. First, it is often necessary for us to refer to a measure of risk on the future cost distribution; hence, we define:
\begin{align*}
v_{t,T}(r,p \, | \, \beta) &= \textnormal{VaR}_{\alpha_t}\bigl[V_{t+1,T}(r,p \, e^{-\kappa_Y} + \psi_{t+1} \, | \, \beta)\bigr],\\
c_{t,T}(r,p \, | \, \beta) &= \textnormal{CVaR}_{\alpha_t}\bigl[V_{t+1,T}(r,p \, e^{-\kappa_Y} + \psi_{t+1} \, | \, \beta)\bigr].
\end{align*}
We will invoke the relationship
\begin{align}
&V_{t,T}(s\,|\,\beta) = \min_{x\in \mathcal X(r)} xp - c_f + \tilde{V}_{t,T}(r+x,p\, | \, \beta),
\label{eq:postdecbellman}
\end{align}
where $\tilde{V}_{t,T}(r+x,p\, | \, \beta) = \rhooperator_{\beta_t} \bigl[ V_{t+1,T}( r+x, p \, e^{-\kappa_Y} + \psi_{t+1} \, | \, \beta) \bigr]$ is called the \emph{post-decision value function}.

 Before continuing, we point out that so far \emph{two} optimization models, the base model (\ref{eq:tradeoffproblem}) and the risk-averse MDP (\ref{eq:rmdp}), have been defined. We re-emphasize that the relationship between these two models is mostly explained in Figure \ref{fig:practical}, where (\ref{eq:rmdp}) serves as a proxy for the more computationally challenging model (\ref{eq:tradeoffproblem}). To be more precise, we are searching for a set of parameters $\beta$ such that the solution to (\ref{eq:rmdp}) is an acceptable solution to (\ref{eq:tradeoffproblem}). This can be thought of as a \emph{policy search} technique, wherein one searches over a subset of candidate policies; here, our candidate policies are solutions to different instances of the risk-averse MDP. As we reviewed in Section \ref{sec:intro}, this overall procedure is used by many authors (see, e.g., \cite{Philpott2012}, \cite{Shapiro2013}, \cite{Cavus2014a}, \cite{Lin2015}, \cite{Maceira2015}, and \cite{Collado2017}). The reason that the risk-averse MDP (\ref{eq:rmdp}) is often not taken to be the base model is due to its lack of practical interpretability. The exception is when $\rho_t = \mathbf{E}_t$, in which case (\ref{eq:rmdp}) is a standard risk-neutral MDP whose objective is simply the expected cumulative cost. In the next section, we focus on analyzing (\ref{eq:rmdp}) and in Section \ref{sec:compat}, we bring the models together and study a specific notion of compatibility between the two.

\section{Structure of the Risk-Averse MDP}
\label{sec:structure}
In this section, we state some structural properties of the risk-averse optimal policy and the associated optimal value function. All proofs that are not included in the body of the paper can be found in Appendix \ref{sec:appendix}.
\begin{restatable}[Convexity]{prop}{propconv}
Let $s = (r,p) \in \mathcal S$ be the state variable. The following statements hold for every $t$ and $\beta$: (i) $xp+ \tilde{V}_{t,T}(r+x,p\, | \, \beta)$ is convex in $x$ on $\mathcal X(r)$, (ii) $\tilde{V}_{t,T}(s\,|\,\beta)$ is convex in the resource $r$, and (iii) $V_{t,T}(s\,|\,\beta)$ is convex in the resource $r$.
\label{prop:conv}
\end{restatable}

The optimal value function also satisfies a monotonicity property, as shown in the following lemma. 

\begin{restatable}[Monotonicity]{prop}{propmonop}
For every $t$ and resource level $r \in \mathcal R \setminus \{\Rmax\}$, the optimal value function $V_{t,T}(r,p\,|\,\beta)$ is increasing in the spot price $p$. 
\label{prop:monop}
\end{restatable}


We now prove that the optimal policy is of the so called \emph{basestock} or \emph{order-up-to} type. In this result, the thresholds are given as a function of the risk parameters $\beta$ and the spot price.

\begin{restatable}[Optimal Policy]{thm}{thmpolicy}
For each time $t<T$ and spot price $p$, there exists a threshold resource level $r_{t,T}(p\,|\,\beta)$ such that the optimal policy takes the following form:
\begin{equation*}
x_{t,T}(r,p \, | \, \beta) = \min \{r_{t,T}(p\,|\,\beta) - r, \xmax\} \, \indicate{r \, \le \, r_{t,T}(p\,|\,\beta)},
\end{equation*}
where $r_{t,T}(p\,|\,\beta)$ is the smallest value of $\tilde{r}$ that minimizes the quantity $\tilde{r} \,p + \tilde{V}_{t,T}(\tilde{r},p \, | \, \beta)$.
\label{thm:optimal_policy}
\end{restatable}
\begin{proof}[Sketch of Proof:]
The proof is standard and follows from convexity of the value function (i.e., Proposition \ref{prop:conv}) along with the fact that the risk measure is coherent (and thus, convex).
\end{proof}

Our next goal is to analyze how $p$ influences the threshold $r_{t,T}(p\,|\,\beta)$.
We are able to show that if the compensation function $\gamma_Y$ is ``not too sensitive'' to the spot price, then the charging thresholds are nonincreasing in the spot price (matching our intuition in that we should charge less when the prices are high). The sensitivty is formalized by the following Lipschitz condition.

\begin{assumption}
The market-dependent compensation function $\gamma_Y$ is positive, increasing, and $L_{\gamma_Y}$-Lipschitz continuous where it holds that $L_{\gamma_Y} \le p_\text{ref}^{-1} \, e^{2\kappa_Y} $.
\label{ass:compensationlip}
\end{assumption}

\begin{remark}
If this assumption is not satisfied, then it is possible to observe strange (i.e., non-monotonic) behavior in the thresholds as a function of price. Note that the value of one unit of charge at the time of customer arrival is governed by $\gamma_Y$. Roughly speaking, the marginal cost we should consider is of the form ``$\text{current price} - \gamma_Y(\text{price deviation at $T+1$})$.'' Assumption \ref{ass:compensationlip} limits the sensitivity of the marginal cost function contributed by $\gamma_Y$.
\end{remark}

Before moving on to the theorem, we develop some notation and state a few technical lemmas.

\begin{restatable}[Lipschitz in Spot Price]{lem}{lemlip} If Assumption \ref{ass:compensationlip} holds, then for any $t$ and $r$, the value function $V_{t,T}(r, \, \cdot \, | \, \beta)$ is Lipschitz continuous in the price dimension.
\label{lem:lip}
\end{restatable}

\begin{restatable}[Differentiability of Value at Risk]{lem}{lemdiffvar}
For every resource state $r \in \mathcal R \setminus \{\Rmax\}$, the value at risk function $v_{t,T}(r,p \, | \, \beta) = \textnormal{VaR}_{\alpha_t}[V_{t+1,T}(r,p \, e^{-\kappa_Y} + \psi_{t+1} \, | \, \beta)]$ is differentiable for almost every spot price $p$.
\label{lem:diffvar}
\end{restatable}

Lemma \ref{lem:lip} and Lemma \ref{lem:diffvar} imply that both $V_{t,T}(r, \, \cdot \, | \, \beta)$ and $v_{t,T}(r,\,\cdot\,|\,\beta)$ are differentiable almost everywhere. It is reasonable for us to remove a null set of non-differentiability (call it $\mathcal V$) from the sample space under consideration for the remainder of the paper, i.e., those $\omega \in \Omega$ where $V_{t,T}(r, P_t(\omega)\,|\,\beta)$ or $v_{t,T}(r,\,\cdot\,|\,\beta)$ are not differentiable. Considering only $\Omega \setminus \mathcal V$ simplifies the discussion in the paper; it is reasonable to do so because $\P(P_t \in \mathcal V) = 0$ for any $t$ (the price process $P_t$ is a continuous random variable and exogenous to the decision process).


\begin{restatable}[Differentiating under the Risk Measure]{lem}{lemsensitivity}
Suppose the value function $V_{t+1,T}( \, \cdot \, | \, \beta)$ is differentiable at $r=r_0 \in \mathcal R \setminus \{\Rmax \}$. Then, we may differentiate under the risk measure in the following sense:
\begin{align*}
\partial_r \tilde{V}_{t,T}(r,p\,|\,\beta) \, |_{r=r_0} = (1&-\lambda) \, \mathbf{E}\bigl[ \partial_r V_{t+1,T}(r,p \, e^{-\kappa_Y} + \psi_{t+1}\,|\,\beta) \bigr] \, |_{r=r_0}\\
& + \lambda \, \mathbf{E}\bigl[ \partial_r V_{t+1,T}(r,p \, e^{-\kappa_Y} + \psi_{t+1}\,|\,\beta) \, | \, \psi_{t+1} \ge \textnormal{VaR}_{\alpha_t}(\psi_{t+1})  \bigr] \, |_{r=r_0}.
\end{align*}
\label{lem:sensitivity}
\end{restatable}
With these lemmas in mind, we are able to show that the charging thresholds obey monotonicity in the spot price, as formalized in the theorem below. Similar theorems for the risk-neutral setting can be found in \cite{Secomandi2010} and \cite{Chen2015b}. Our analysis extends these results (although under a different MDP model) to the risk-averse setting by utilizing the preceding lemmas.
\begin{restatable}[Nonincreasing Threshold in Spot Price]{thm}{thmthreshspot}
Under Assumption \ref{ass:compensationlip}, the threshold $r_{t,T}(p\,|\,\beta)$ is nonincreasing in the spot price $p$ for any degree of dynamic risk-aversion $\beta$.
\label{thm:threshspot}
\end{restatable}


\begin{remark}
Theorem \ref{thm:threshspot} is proven here for a spot price model that contains mean-reversion, seasonality, and jumps, three properties that are most appropriate for electricity spot prices. However, we remark that the result can be adapted to other reasonable price models as well. The crucial property used in the proof of Theorem \ref{thm:threshspot} is that $p - \E[P_{t+1} \, | \, P_t = p]$ is nondecreasing in $p$. For example, our framework can incorporate general models of the form $P_{t+1} =  \zeta_{t+1}^1 \, P_t + \zeta_{t+1}^2$ for some random variables $\zeta_{t+1}^1$ and $\zeta_{t+1}^2$, under the condition that $\E[  \zeta_{t+1}^1] \le 1$. Several popular choices, e.g., AR(1), martingale, and i.i.d. noise terms (along with any seasonality function), can be represented under this general model. Note that Assumption \ref{ass:compensationlip} is specialized for the model given in (\ref{eq:discretetimeou}) and would need to be altered under a different spot price model.
\end{remark}

\section{Practical Risk Compatibility}
\label{sec:compat}
The last section analyzes the risk-averse optimal policy under a \emph{given} dynamic risk measure $\rhooperator_{0,T}(\,\cdot\,|\,\beta)$. This is a standard approach in papers that consider risk-averse models; see, e.g., \cite{Lin2015}, \cite{Devalkar2016}, where the structure of the risk-averse optimal policy is explicitly characterized. However, in this section, we are interested in going one step further and commenting on the protocol of Figure \ref{fig:practical} in the context of our EV charging problem. From this point of view, the risk-averse MDP model is simply a computational tool used to generate candidate policies that are then evaluated under a practical risk metric and a practical reward metric (the base model objectives). The motivation for doing so is that directly optimizing the practical risk and reward objectives can be computationally difficult, while the risk-averse MDP enjoys a dynamic programming decomposition.

\subsection{Defining Risk Compatibility}
When $\pi = \pi_T(\beta)$, we will use the convenient shorthand $ \mathbf{S}^\beta_T=\mathbf{S}^\pi_T$ for the vector of visited states.
We are interested in defining a notion of \emph{compatibility} between the potentially divergent views of the problem.  A new version of Figure \ref{fig:practical} updated with mathematical notation is given in Figure \ref{fig:practical2}. Before continuing, we need to first define the notion of ``more risk-aversion.''
\begin{figure}[h]
\vspace{-10pt}
        \centering
        \includegraphics[width=0.85\textwidth]{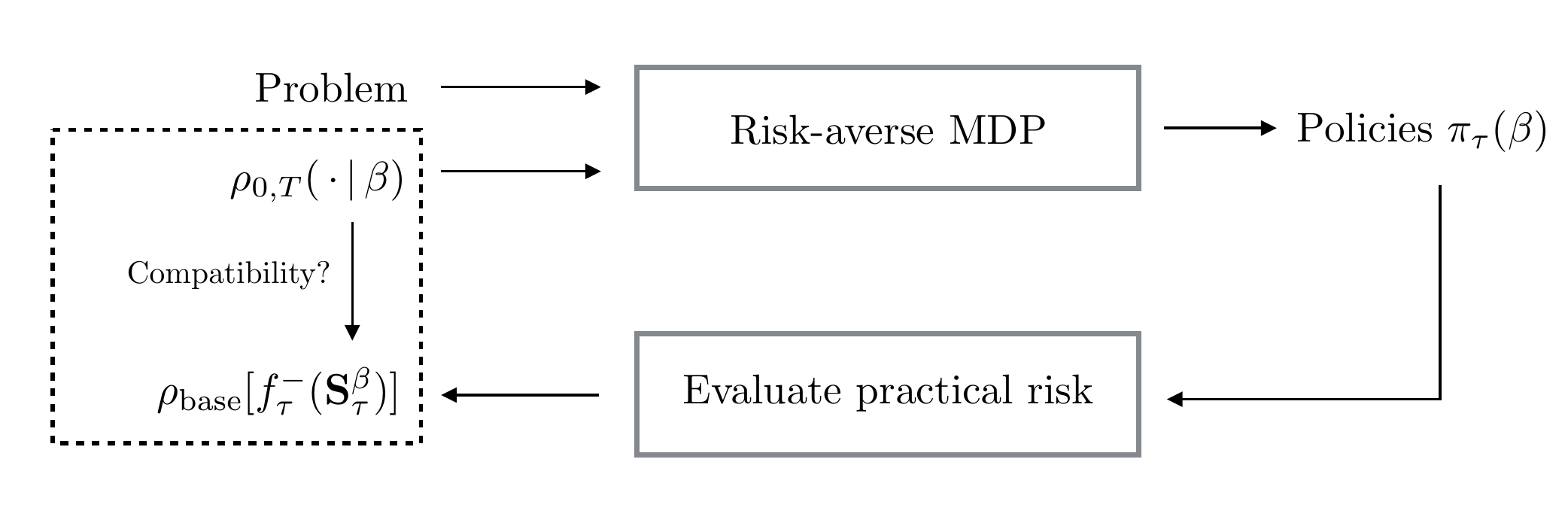}
                \vspace{-10pt}
        \caption{Evaluation using Practical Metrics of Risk and Reward}
        \label{fig:practical2}
\end{figure}

\begin{definition}[Degree of Dynamic Risk-Aversion]
Consider two dynamic risk measures $\rhooperator_{0,T}$ and $\rhooperator_{0,T}'$. If $\rhooperator_{0,T}(\mathbf{C}) \le \rhooperator_{0,T}'(\mathbf{C})$ for any cost sequence $\mathbf{C}$, then $\rhooperator_{0,T}'$ is \emph{more risk-averse} than $\rhooperator_{0,T}$. 
\label{def:degreerisk}
\end{definition}

To give a concrete example in the single period case, \cite{Belles-Sampera2014} points out that when CVaR is used instead of VaR in the financial industry (note that $\textnormal{VaR}_\alpha(X) \le \textnormal{CVaR}_\alpha(X)$ for any cost $X$), then considerably higher capital reserves are required to offset the risk. It is thus reasonable to say that $\textnormal{CVaR}_\alpha$ is more risk-averse than $\textnormal{VaR}_\alpha$.

\begin{definition}[Practical Risk Compatibility]
Let the dynamic risk measure be parameterized by $\beta$ that varies within some set $\mathcal B$. Suppose that there is a partial order $\le$ on $\mathcal B$ such that for $\beta \le \beta'$, the dynamic risk measure $\rhooperator_{0,T}(\,\cdot\,|\,\beta')$ is \emph{more risk-averse} than $\rhooperator_{0,T}(\,\cdot\,|\,\beta)$. We say that $\rhooperator_{0,T}(\,\cdot\,|\,\beta)$ has the property of \emph{practical compatibility on $\mathcal B$ with respect to
$f^-_T(\, \cdot \,)$} if the mapping $\beta \rightarrow f^-_T(\mathbf{S}^\beta_T)$ is nonincreasing for any $T$.
\end{definition}
A consequence of this definition is that $\beta \rightarrow \rho_\text{base} [ f^-_\tau(\mathbf{S}^\beta_\tau)]$ is nonincreasing. In other words, compatibility means that ``more risk-aversion'' in the dynamic risk measure sense should lead to lower risk in the practical sense as well. If, for example, more risk-aversion leads to higher practical risk, then we may conclude that the risk-averse MDP is incompatible with our practical objective and choose an alternative optimization procedure. The partial order for determining the degree of dynamic risk-aversion (in the sense of Definition \ref{def:degreerisk}) is generally easy to determine; for instance, $\textnormal{CVaR}_{\alpha}(X)$ is interpreted as more risk-averse as $\alpha$ increases \citep{Chen2009}, and $(1-\lambda) \, \mathbf{E}(X) + \lambda \,\textnormal{CVaR}_{\alpha}(X)$ is more risk-averse as $\lambda$ increases \citep{Philpott2012}.

Note that the notion of compatibility depends on many factors: we must first define the dynamic risk measure $\rhooperator_{0,T}(\,\cdot\,|\,\beta)$, the parameter set $\mathcal B$ along with the partial order $\le$, and the practical risk metric. Hence, for any given problem, we may observe compatibility in some cases and incompatibility in others (incompatibility refers to the situation when there are regions of $\mathcal B$ when the mapping $\beta \rightarrow \rho_\text{base}\,[f^-_\tau(\mathbf{S}^\beta_\tau)]$ increases). An example can be found in \cite{Maceira2015}, where it was observed that increasing risk-aversion in the dynamic risk measure unexpectedly \emph{increased} the practical risk (energy not supplied):
\begin{displayquote}
It is possible to see that for some combinations of parameters [$\alpha$ and $\lambda$] both operation cost and EENS [energy not supplied] increase simultaneously. Even though we expect higher operation costs when applying the CVaR mechanism, an increase in the EENS caused by preventive load curtailment is not desirable. Therefore, we call the region containing such combinations as ``undesirable region.'' On the other hand, there are combinations of parameters that provide results located in a desirable area, i.e., cases where the operation cost increases and the EENS decreases.
\end{displayquote}
Given that the system in \cite{Maceira2015} was operationalized in Brazil, the issue of compatibility is a critical one. Another example of incompatibility can be observed in Figure 11 of \cite{Shapiro2013}. A stylized illustration of these examples is given in Figure \ref{fig:compatibility}.

\begin{figure}[h]
        \centering
        \begin{subfigure}[b]{0.31\textwidth}
                \centering
                \includegraphics[width=\textwidth]{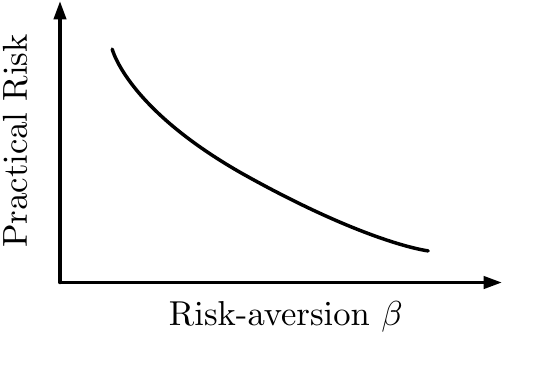}
                \caption{Compatible}
        \end{subfigure}
        \begin{subfigure}[b]{0.31\textwidth}
                \centering
                \includegraphics[width=\textwidth]{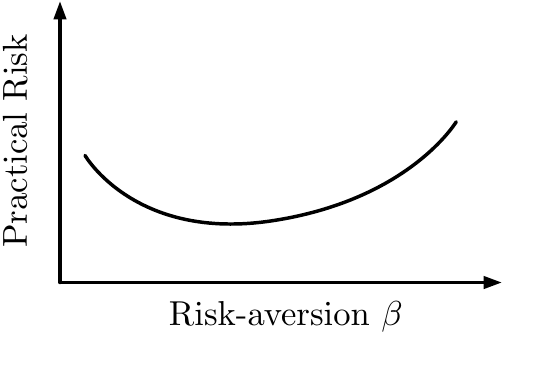}
                \caption{Incompatible}
        \end{subfigure}
        \begin{subfigure}[b]{0.31\textwidth}
                \centering
                \includegraphics[width=\textwidth]{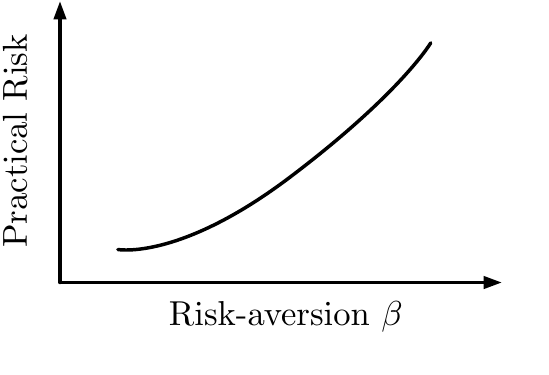}
                \caption{Incompatible}
        \end{subfigure}
        \caption{Illustration of Practical (In)compatibility}
        \label{fig:compatibility}
\end{figure}

\subsection{Compatibility Results} In the remainder of this section, we state and prove several risk compatibility properties under the condition that for any $T$, the function $f^-_T( \mathbf{S}_T^\pi)$ is componentwise nonincreasing in the resource states $\mathbf{R}_T^\pi = (R_0^\pi, R_1^\pi, \ldots, R_T^\pi)$, a condition we refer to as \emph{resource-monotonicity}. In other words, more energy in the vehicle corresponds to lower practical risk. Although we did not require resource-monotonicity when introducing practical risk metrics (in the interest of potential generalizations to other models), the inherent trade-offs associated with the ``charging versus waiting'' decision studied in the dynamic charging problem necessitate this property. We start with an important lemma.

\begin{restatable}[Decreasing Marginal Value in Price]{lem}{marginalval}
Under Assumption \ref{ass:compensationlip}, the pre-decision and post-decision marginal value of energy, $\partial_r V_{t,T}(r,p\,|\,\beta)$ and $\partial_r \tilde{V}_{t,T}(r,p\,|\,\beta)$, are nonincreasing in the spot price $p$ for any fixed risk-aversion $\beta$.
\label{lem:marginalval}
\end{restatable}

We remind the reader that in the proof of Theorem \ref{thm:threshspot}, we showed that $p+\partial_r \tilde{V}_{t,T}(r,p\,|\,\beta)$ is \emph{nondecreasing} in $p$. The combined effect with the above lemma implies that $\partial_r \tilde{V}_{t,T}(r,p\,|\,\beta)$ cannot decrease too quickly, but is nonincreasing nonetheless. Leveraging this lemma, we are able to show our first compatibility result for the case when the dynamic risk measure is parameterized by $\beta_0 = (\lambda_0, \alpha_0)$. Recall that we described two parameter regimes, the general charging regime where $\xmax$ is unrestricted and the fast charging regime where $\xmax$ is larger than $\Rmax$.

\begin{restatable}[Compatibility in General Charging Regime]{thm}{thmcompatone}
Fix any reservation horizon $T$ and the risk-aversion $\beta_{-0}^* = (\lambda_1^*, \alpha_1^*, \ldots, \lambda_{T}^*, \alpha_{T}^*)$ for all periods except the first. Let 
\[
\mathcal B(\beta_{-0}^*) = \{\beta = (\lambda_0, \alpha_0, \lambda_1^*, \alpha_1^*, \ldots, \lambda_{T}^*, \alpha_{T}^*) : \lambda_0 \in [0,1], \, \alpha_0 \in (0,1)\}
\]
be the set of risk-aversion parameters where $(\lambda_t, \alpha_t)$ matches $(\lambda_t^*, \alpha_t^*)$ of $\beta_{-0}^*$ while the risk-aversion in the first period is a ``free parameter.'' Taking the partial order $\le$ to be the componentwise inequality on $\mathbb R^2$ gives compatibility on the set $\mathcal B(\beta_{-0}^*)$ between $\rhooperator_{0,T}(\,\cdot\,|\,\beta)$ and any resource-monotone $f_T^-$.
\label{thm:thmcompatone}
\end{restatable}

We remark that the above compatibility result requires a very specific parameterization of the dynamic risk measure where only the initial risk-aversion parameters are free. This restrictive parameterization is not ideal in that it limits the decision-maker to a small set of risk-averse policies. A key takeaway from Theorem \ref{thm:thmcompatone} and its narrow scope highlights the delicate nature of the practical risk compatibility question.

Interestingly, we now show in Theorem \ref{thm:thmcompattwo} that compatibility \emph{can be shown with a fully general parameterization of the dynamic risk measure} if we focus only on the fast charging regime, where $\Rmax \le \xmax$. We include the proof of the theorem here in order to comment on the critical differences between the fast charging regime and the general regime described in Theorem \ref{thm:thmcompatone} (see Section \ref{sec:incompat}).

\begin{restatable}[Compatibility in Fast Charging Regime]{thm}{thmcompattwo}
Fix any reservation horizon $T$ and consider the fully general risk-aversion parameter set 
\[
\mathcal B = \{\beta : \lambda_t \in [0,1], \, \alpha_t \in (0,1), \, t= 0, 1, \ldots, T \}
\]
of risk-aversion parameters. Let the partial order $\le$ be the componentwise inequality on $\mathbb R^T$. Then, compatibility holds on the set $\mathcal B$ between $\rhooperator_{0,T}(\,\cdot\,|\,\beta)$ and any resource-monotone $f_T^-$.
\label{thm:thmcompattwo}
\end{restatable}

\begin{proof}
Consider the risk-aversion parameters $\beta, \, \beta' \in \mathcal B$ where
\[
\beta = (\lambda_0, \alpha_0, \ldots, \lambda_{T}, \alpha_{T}) \le  (\lambda_0', \alpha_0', \ldots, \lambda_{T}', \alpha_{T}') = \beta'.
\]
Suppose that $r_{t,T}(p\,|\,\beta) \le r_{t,T}(p\,|\,\beta')$ for any spot price $p$ and time $t$. If this is true, then since the thresholds (i.e., charging ``targets'') are higher for every $t$, it follows that $\mathbf{R}_T^\beta \le \mathbf{R}_T^{\beta'}$. Therefore, threshold monotonicity $r_{t,T}(p\,|\,\beta) \le r_{t,T}(p\,|\,\beta')$ is a sufficient condition for the compatibility criterion $f_T^-(\mathbf{S}_T^{\beta'}) \le f_T^-(\mathbf{S}_T^\beta)$.

Our goal for the remainder of this proof is to show that the marginal value of energy \emph{decreases} as risk-aversion increases: $\partial_r \tilde{V}_{t,T}(r,p\,|\,\beta) \ge \partial_r \tilde{V}_{t,T}(r,p\,|\,\beta')$. This would imply the desired condition $r_{t,T}(p\,|\,\beta) \le r_{t,T}(p\,|\,\beta')$, from which the reasoning above can be applied to complete the argument.

\vspace{7pt} \noindent \textbf{Base Case.}  We proceed via backward induction, starting with the base case of $t=T-1$. First, note that $\gamma_Y$ is positive by Assumption \ref{ass:compensationlip} and $\rhooperator_{\beta_T}$ is nondecreasing in $\beta_T$. The derivative 
\begin{align*}
\partial_r V_{t+1,T}(r,p \, &e^{-\kappa_Y} + \psi_{t+1}\,|\,|beta)\\
&= -p_\text{ref} \bigl[ 1 + 2 \gamma_h \, h(r) +  \rhooperator_{\beta_T} [\gamma_Y((p \, e^{-\kappa_Y} + \psi_{t+1}) \, e^{-\kappa_Y} + \psi_{t+1,Y})] \bigr] 
\end{align*}
is thus nonincreasing in $\beta_T$. Next, by Lemma \ref{lem:marginalval} and the fact that $\textnormal{VaR}_{\alpha_{T-1}}(\psi_{T})$ is increasing in $\alpha_{T-1}$, it follows that the mapping
 \[
 \alpha_{T-1} \mapsto \mathbf{E}\bigl[ \partial_r V_{T,T}(r,p \, e^{-\kappa_Y} + \psi_{T}\,|\,\beta) \, | \, \psi_{T} \ge \textnormal{VaR}_{\alpha_{T-1}}(\psi_{T})   \bigr]
 \]
 is nonincreasing. Hence, we have
  \begin{align*}
\partial_r \tilde{V}_{T-1,T}(r,p\,|\,\beta) &=  \begin{aligned}[t](1 -\lambda_{T-1}) \, &\mathbf{E}\bigl[ \partial_r V_{T,T}(r,p \, e^{-\kappa_Y} + \psi_{T}\,|\,\beta) \bigr] \\
&+ \lambda_{T-1} \, \mathbf{E}\bigl[ \partial_r V_{T,T}(r,p \, e^{-\kappa_Y} + \psi_{T}\,|\,\beta) \, | \, \psi_{T} \ge \textnormal{VaR}_{\alpha_{T-1}}(\psi_{T})   \bigr]\end{aligned}\\
&\ge \begin{aligned}[t](1 -\lambda_{T-1}') \, &\mathbf{E}\bigl[ \partial_r V_{T,T}(r,p \, e^{-\kappa_Y} + \psi_{T}\,|\,\beta') \bigr] \\
&+ \lambda_{T-1}' \, \mathbf{E}\bigl[ \partial_r V_{T,T}(r,p \, e^{-\kappa_Y} + \psi_{T}\,|\,\beta') \, | \, \psi_{T} \ge \textnormal{VaR}_{\alpha_{T-1}'}(\psi_{T})   \bigr]\end{aligned}\\
&= \partial_r \tilde{V}_{T-1,T}(r,p\,|\,\beta'),
\end{align*}
completing the base case. The induction hypothesis is $\partial_r \tilde{V}_{t,T}(r,p\,|\,\beta) \ge \partial_r \tilde{V}_{t,T}(r,p\,|\,\beta')$ and we will verify that the same statement holds when $t-1$ replaces $t$.

\vspace{7pt} \noindent  \textbf{Inductive Step.} 
Let $P = p \, e^{-\kappa_Y} + \psi_{t}$. We now consider the value function at time $t$ evaluated at $P$ (i.e., from the perspective of time $t-1$):
\begin{equation}
V_{t,T}(r,p\,|\,\beta) = x_{t,T}(r,p \, | \, \beta) \cdot P - c_f + \tilde{V}_{t,T}(r+x_{t,T}(r,p \, | \, \beta) ,P\,|\,\beta),
\label{eq:valuefunc}
\end{equation}
and as a first step, we aim to show that $\partial_r V_{t,T}(r,p\,|\,\beta) \ge \partial_r V_{t,T}(r,P\,|\,\beta')$. Since we are in the fast charging regime, the optimal policy at time $t$ is to either charge (all the way) up to the threshold $r_{t,T}(p\,|\,\beta)$ or do nothing if the charge level is already above $r_{t,T}(p\,|\,\beta)$. We define the price sets 
\[
\mathscr P_\text{low}(r,\beta) = \{P: r \le r_{t,T}(p\,|\,\beta)\} \quad \text{and} \quad \mathscr P_\text{high}(r,\beta) = \{P:r_{t,T}(p\,|\,\beta) < r\}.
\]
Due to $r_{t,T}(p\,|\,\beta)$ being nonincreasing in $p$, these price sets are disjoint and monotone in the sense that any element of $\mathscr P_\text{low}(r,\beta)$ is no greater than any element of $\mathscr P_\text{high}(r,\beta)$. Previously, in the proof of Theorem \ref{thm:threshspot} (in Appendix \ref{sec:appendix}), the risk-aversion was fixed so the price sets did not need to explicitly depend on 
$\beta$; however, now that we are studying the impact of risk-aversion, it is crucial to analyze this dependence. See Figure \ref{fig:pricesets_fast} for an illustration of the price sets: the two horizontal axes represent the price $P$ and the vertical axis is the resource value $r$. The diagram shows that (by the induction hypothesis) the charging threshold $r_{t,T}(p\,|\,\beta)$ at $t$ is nondecreasing as we move from $\beta$ (the black curve) to $\beta'$ (the red curve). This, in turn, alters the low and high price sets shown on the bottom horizontal axis for $\beta$ (in black) and the top horizontal axis for $\beta'$ (in red).

\begin{figure}[h]
        \centering
        \includegraphics[width=\textwidth]{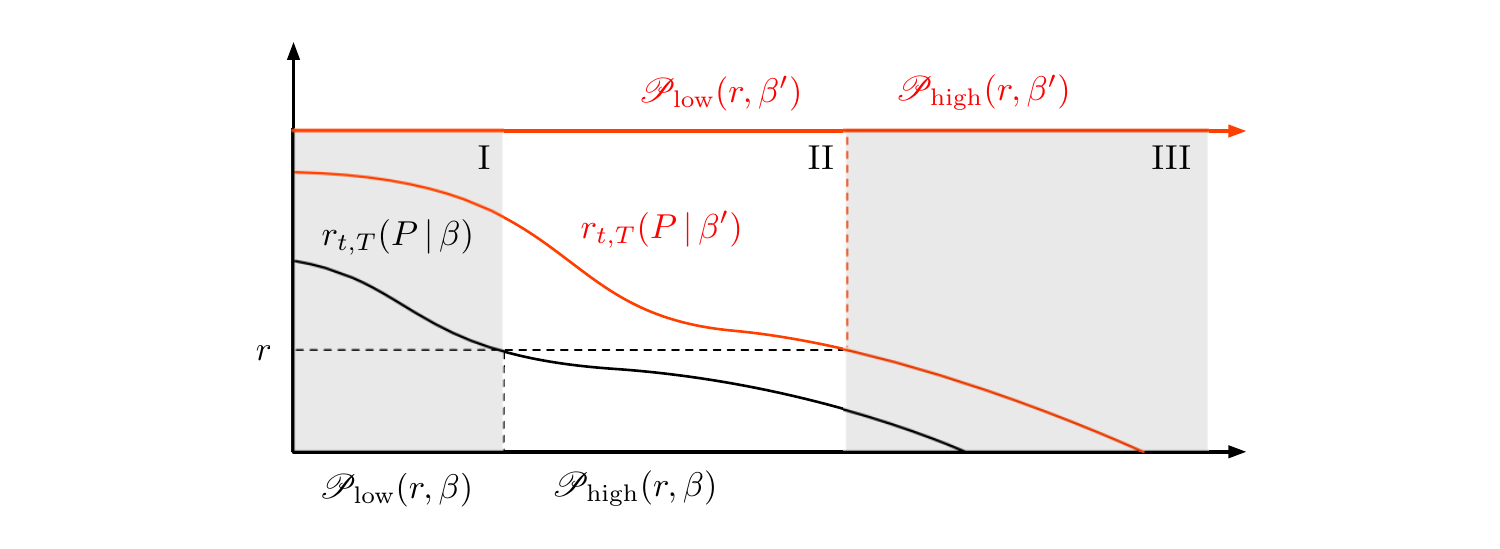}
        \caption{Illustration of the Price Sets for the Fast Charging Case}
        \label{fig:pricesets_fast}
\end{figure}

Doing casework on each of the spot price intervals, we can rewrite (\ref{eq:valuefunc}) for a fixed $r$ as a piecewise function of $P$: 
\begin{equation*}
V_{t,T}(r,p\,|\,\beta) = \begin{cases}
(r_{t,T}(p\,|\,\beta) - r ) \cdot P -c_f+  \tilde{V}_{t,T}(r_{t,T}(p\,|\,\beta),P\,|\,\beta) &\text{ if } P \in \mathscr P_\text{low}(r,\beta),\\
-c_f+\tilde{V}_{t,T}(r,P\,|\,\beta) &\text{ if } P \in \mathscr P_\text{high}(r,\beta).\\
\end{cases}
\end{equation*}
Differentiating in the resource $r$, we obtain
\begin{equation}
\partial_r V_{t,T}(r,p\,|\,\beta) = \begin{cases}
-P  &\text{ if } P \in \mathscr P_\text{low}(r,\beta),\\
\partial_r \tilde{V}_{t,T}(r,p\,|\,\beta) &\text{ if } P \in \mathscr P_\text{high}(r,\beta).\\
\end{cases}
\label{eq:dr_fast}
\end{equation}
As shown in Figure \ref{fig:pricesets_fast}, when comparing $\beta$ and $\beta'$, there are three cases (I, II, and III) that we need to consider. In Case I, we have $P \in \mathscr P_\text{low}(r,\beta) \cap \mathscr P_\text{low}(r,\beta')$ and it is clear by (\ref{eq:dr_fast}) that $\partial_r V_{t,T}(r,p\,|\,\beta) = \partial_r V_{t,T}(r,P\,|\,\beta') = -P$. Case III is similarly straightforward because when $P \in \mathscr P_\text{high}(r,\beta) \cap \mathscr P_\text{high}(r,\beta')$, we know that by the induction hypothesis
\[
\partial_r V_{t,T}(r,p\,|\,\beta) = \partial_r \tilde{V}_{t,T}(r,p\,|\,\beta) \ge \partial_r \tilde{V}_{t,T}(r,p\,|\,\beta') = \partial_r V_{t,T}(r,P\,|\,\beta').
\]
Now let us move to Case II, when $P \in \mathscr P_\text{high}(r,\beta) \cap \mathscr P_\text{low}(r,\beta')$. Observe the following important property whenever $P \in \mathscr P_\text{high}(r,\beta) = \{P:r_{t,T}(p\,|\,\beta) < r\}$. As characterized in Theorem \ref{thm:optimal_policy}, the threshold $\tilde{r} = r_{t,T}(p\,|\,\beta)$ minimizes the quantity  $\tilde{r} \,P + \tilde{V}_{t,T}(\tilde{r},p \, | \, \beta)$. By convexity of $\tilde{r} \,P + \tilde{V}_{t,T}(\tilde{r},p \, | \, \beta)$ in $\tilde{r}$ (Proposition \ref{prop:conv}), it follows that the derivative $P + \partial_r \tilde{V}_{t,T}(\tilde{r},p \, | \, \beta)$ evaluated at some $\tilde{r} > r_{t,T}(p\,|\,\beta)$ must be nonnegative. Hence, we have for Case II:
\[
\partial_r V_{t,T}(r,p\,|\,\beta) = \partial_r \tilde{V}_{t,T}(r,p\,|\,\beta) \ge -P = \partial_r V_{t,T}(r,P\,|\,\beta')
\]
which completes the argument that $\partial_r V_{t,T}(r,p\,|\,\beta) \ge \partial_r V_{t,T}(r,P\,|\,\beta')$ for any $P$. Because $r$ was arbitrarily fixed, the inequality holds for any $r$ as well. To finish the inductive step, i.e. $\partial_r \tilde{V}_{T-1,T}(r,p\,|\,\beta) \ge \partial_r \tilde{V}_{T-1,T}(r,p\,|\,\beta')$, the proof of the base case can be essentially repeated.
\end{proof}

\subsection{Incompatibilty}
\label{sec:incompat}
What goes wrong in the slow charging case if we wish to employ the fully general $\mathcal B$ from Theorem \ref{thm:thmcompattwo}? As with logic introduced in the proof of Theorem \ref{thm:threshspot} in Appendix \ref{sec:appendix} (see Figure \ref{fig:pricesets}), extending the analysis of Theorem \ref{thm:thmcompattwo} to the slow charging case also requires the definition of \emph{three} price sets: $\mathscr P_\text{low}(r,\beta) = \{P: r+\xmax < r_{t,T}(p\,|\,\beta)\}$, $ \mathscr P_\text{med}(r,\beta) = \{P: r \le r_{t,T}(p\,|\,\beta) \le r+\xmax\}$, and  $\mathscr P_\text{high}(r,\beta) = \{P:r_{t,T}(p\,|\,\beta) < r\}$.

\begin{figure}[h]
        \centering
        \includegraphics[width=\textwidth]{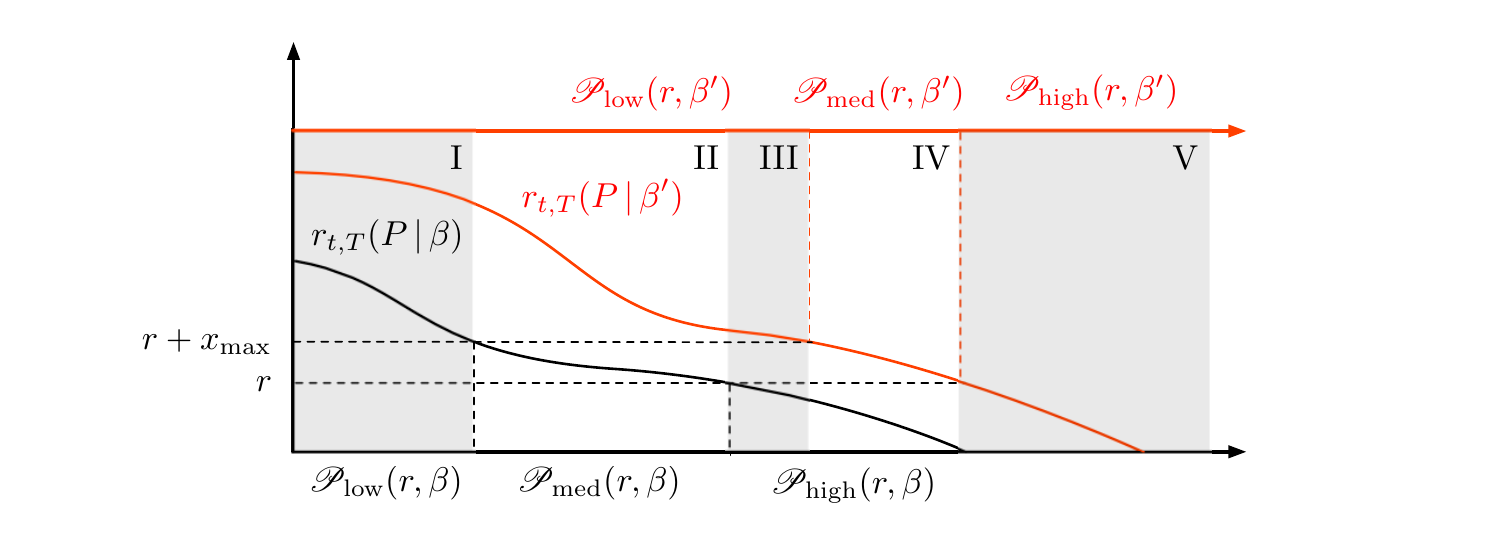}
        \caption{Illustration of the Price Sets for the Slow Charging Case}
        \label{fig:pricesets_slow}
\end{figure}

In Figure \ref{fig:pricesets_slow}, we illustrate a situation where $\mathscr P_\text{high}(r,\beta) \cap \mathscr P_\text{low}(r,\beta')  \not = \varnothing$; this region is denoted ``III.'' The two quantities that we need to compare in this region to show $\partial_r V_{t,T}(r,p\,|\,\beta) \ge \partial_r V_{t,T}(r,P\,|\,\beta')$ are $\partial_r \tilde{V}_{t,T}(r,p\,|\,\beta)$ and $\partial_r \tilde{V}_{t,T}(r+\xmax,P\,|\,\beta')$ --- unfortunately, it is \emph{not necessarily true} that the former is greater than the latter. Note, however, that Case I, II, IV, and V have the correct inequality. One can confirm via numerical experiments that incompatibility over $\mathcal B$ in the slow charging case indeed exists.



\section{Selecting the Degree of Dynamic Risk-Aversion}
\label{sec:select}
So far, we have studied the dynamic EV charging problem along two dimensions: (1) we provide an in-depth characterization on the structure of the optimal policy for a \emph{fixed} degree of risk-aversion $\beta$ in Section \ref{sec:structure}, and (2) we consider \emph{varying} $\beta$ and analyze the issue of compatibility between optimal policies of risk-averse MDPs (specified under a dynamic risk measure built by nesting mean-CVaR) and resource-monotone practical risk metrics in Section \ref{sec:compat}. Now that the approach shown in Figures \ref{fig:practical} and \ref{fig:practical2} has been justified (to an extent) by the theorems of Section \ref{sec:compat}, one question still remains: \emph{how do we select $\beta$ so that the optimal risk-averse MDP policy performs ``well'' under the base model?}

As we previously discussed, the current practice is to heuristically select a large batch of potential $\beta$'s, solve a large number of MDPs, simulate a large number of resulting policies, and then choose one that provides an acceptable trade-off (see, e.g., \cite{Maceira2015}). Unfortunately, this heuristic tuning method requires significant computational resources, negating the original motivation to solve practical reward versus risk problem via the simple and computationally effective dynamic programming approach (refer to Figure \ref{fig:practical}). In addition to CPU time, another weakness of the heuristic tuning approach is that the final choice of $\beta$ must be from the pre-selected batch. 

We now propose a simple approximation procedure that mitigates these issues. It utilizes structural information from our compatibility results so that $\beta$ can be selected in a more principled manner. Suppose that we are given practical metrics $r_\text{base}\,[f^+_\tau(\mathbf{S}^\pi_\tau)]$ and $\rho_\text{base}\,[f^-_\tau(\mathbf{S}^\pi_\tau)]$, along with a set $\bar{\mathcal B} \subseteq \mathcal B$ and $\le$ that are compatible with $f^-_T$ for all $T$. Given a ``trade-off parameter'' $\varepsilon$, our goal is to search for some $\beta \in \bar{\mathcal B}$ such that the practical risk $\rho_\text{base}\,[f^-_\tau(\mathbf{S}^\beta_\tau)] \le \varepsilon$ while $r_\text{base}\,[f^+_\tau(\mathbf{S}^\pi_\tau)]$ is maximized; see the formulation given in (\ref{eq:tradeoffproblem}).


\begin{description}
\item[Step 1 \textnormal{(Solve MDPs).}] Sample a small set of risk-aversion parameters $\{\beta^1, \beta^2, \ldots, \beta^n\}$ and solve the associated MDPs for each $T$. For every $i$, use simulation to obtain estimates $y^i \approx r_\text{base}\,[f^+_\tau(\mathbf{S}^{\beta^i}_\tau)]$ and $z^i \approx \rho_\text{base}\,[f^-_\tau(\mathbf{S}^{\beta^i}_\tau)]$.
\item[Step 2 \textnormal{(Structured Regression).}] Utilize the ``datasets'' of practical reward/risk versus $\beta$
\[
\{(\beta^1, y^1), (\beta^2, y^2), \ldots, (\beta^n, y^n)  \} \quad \text{and} \quad \{(\beta^1, z^1), (\beta^2, z^2), \ldots, (\beta^n, z^n)  \}
\]
to construct regression models $\hat{r}_\text{base}(\beta)$ and $\hat{\rho}_\text{base}(\beta)$ that respectively approximate the practical reward and risk incurred by the MDP policies specified by $\beta$. Leverage the structural compatibility result to complement the regression procedure (see below).
\item[Step 3 \textnormal{(Policy Selection).}] Use the models $\hat{r}_\text{base}(\beta)$ and $\hat{\rho}_\text{base}(\beta)$ to find a parameter $\hat{\beta}$ such that the practical reward $\hat{r}_\text{base}(\hat{\beta})$ is maximized while the practical risk is constrained: $\hat{\rho}_\text{base}(\hat{\beta}) \le \varepsilon$. Solve the MDPs associated with $\hat{\beta}$ and implement the policy.
\end{description}

The main idea of Step 2 is to take advantage of Theorems \ref{thm:thmcompatone} and \ref{thm:thmcompattwo} and employ a form of structured regression to ensure that $\hat{\rho}_\text{base}(\beta)$ is nonincreasing in $\beta$. Many monotone regression techniques are specified only for the one dimensional case \citep{Mukerjee1988,Dette2006}, but following the general idea given in \cite{Ahmadi2014} for convex regression, if we restrict $\hat{\rho}_\text{base}$ to take the form of a polynomial, then the monotonicity constraints can be approximated with \emph{sum of squares} (SOS) constraints. The resulting problem, for an $l_1$ error function, can be solved efficiently using semi-definite programming \citep{Parrilo2003}. For example, suppose we are using the set $\bar{\mathcal B} = \{\beta : (\lambda,\alpha)= \beta_0 = \beta_1 = \cdots =\beta_{T}\}$, the set of risk parameters that are time-independent. This set can be parameterized by two scalars, $\lambda$ and $\alpha$. Letting $\beta^i = (\lambda^i, \alpha^i)$, the structured regression problem is:
\begin{equation*}
\begin{aligned}
& \underset{\text{poly } \hat{\rho}_\text{base} }{\text{minimize}}
& & \sum_{i=1}^n | \, \hat{r}_\text{base}(\lambda^i, \alpha^i) - z^i \,\bigr|\\
& \text{subject to}
& & -\partial_\lambda \,\hat{r}_\text{base}(\lambda, \alpha), \, -\partial_\alpha\, \hat{r}_\text{base}(\lambda, \alpha) \text{ are SOS}.
\end{aligned}
\end{equation*}
We illustrate this procedure using numerical experiments in Section \ref{sec:num}.

\section{Numerical Results and Case Study}
\label{sec:num}
In this section, we investigate our results on a realistic instance of the problem in the presence of a DC fast charging station (e.g., the ones described in Section \ref{sec:intro} by Tesla and ChargePoint). Specifically, we empirically verify the structural results of Theorem \ref{thm:optimal_policy} and Theorem \ref{thm:threshspot}, in addition to the compatibility result (for the fast charging case) of Theorem \ref{thm:thmcompattwo}. We then show the effectiveness of the risk-averse MDP selection procedure described in Section \ref{sec:select}. The  stochastic models of the price process $\{P_t\}$ and the reservation length $\tau$ are informed by market data from CAISO and parking garage inventory data from the city of Santa Monica, California.

\begin{figure}[h]
        \centering
        \includegraphics[width=0.8\textwidth]{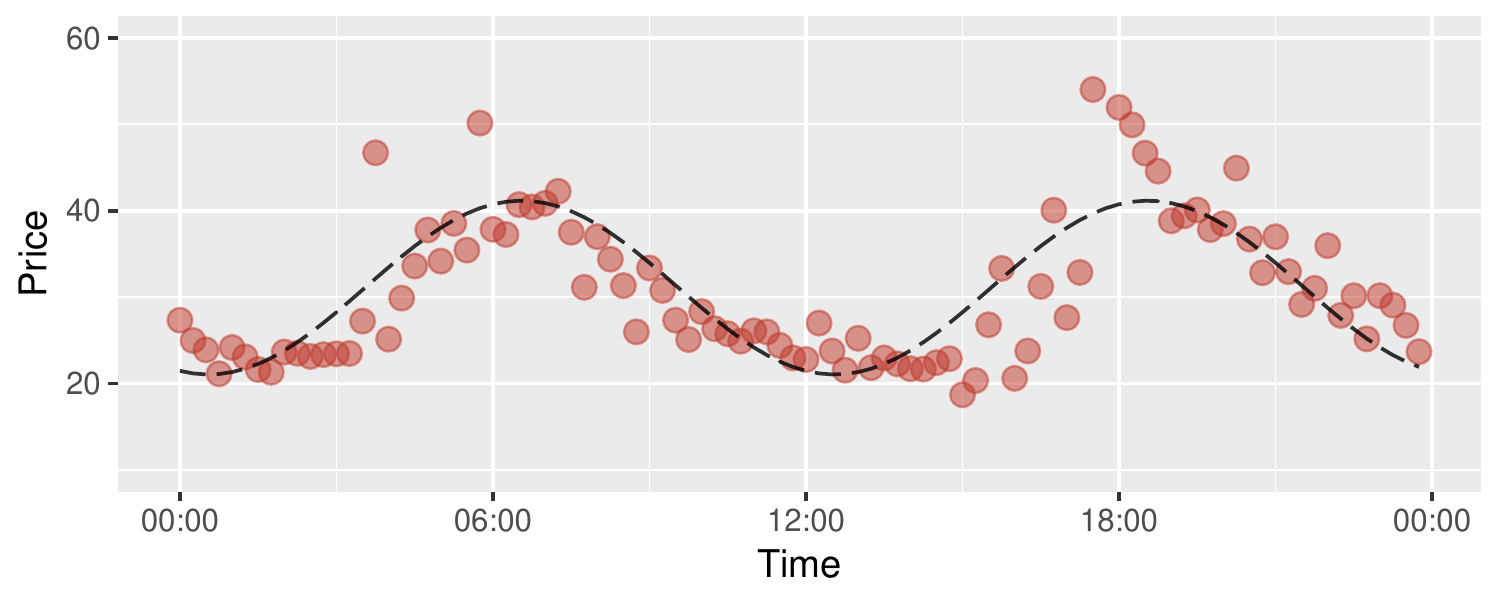}%
        \vspace{-5pt}
        \caption{Seasonality Function $g(t)$ and Average (Across Time) Prices}
                \vspace{-10pt}

        \label{fig:seas}
\end{figure}

\subsection{Stochastic Model \& Problem Parameters}
The spot market in our model is based on a dataset from the CAISO fifteen minute market (the \texttt{BAYSHOR2\char`_1\char`_N001} node) during the winter period December 1, 2016 to February 28, 2017. Since our problem is at the time scale of a few hours at most, we are primarily interested in understanding the \emph{intra-day} seasonality. To accomplish this, for each 15 minute period during the day, we compute the average price across all weekdays in the dataset and then fit a sinusoidal function (with a period of one day) to the resulting time series; see Figure \ref{fig:seas}. The estimated seasonality function is given by
\[
g(t) = 13.586 \, \sin(2\pi t/48) -0.7597 \, \cos(2 \pi t/48) + 34.1362,
\]
where $t$ measured in units of 15 minutes. 

After removing the seasonality, we used maximum likelihood estimation to find the parameters of the $Y_t$ stochastic process, resulting in $\mu_Y = -0.492$, $\kappa_Y = 0.341$, $\sigma_Y = 5.350$, $\mu_J = -0.484$, $\sigma_J = 40.602$, and $\lambda = 0.131$. Figure \ref{fig:combined} provides a comparison between actual and fitted prices that includes both $g(t)$ and $Y_t$. 

\begin{figure}[h]
        \centering
        \includegraphics[width=0.8\textwidth]{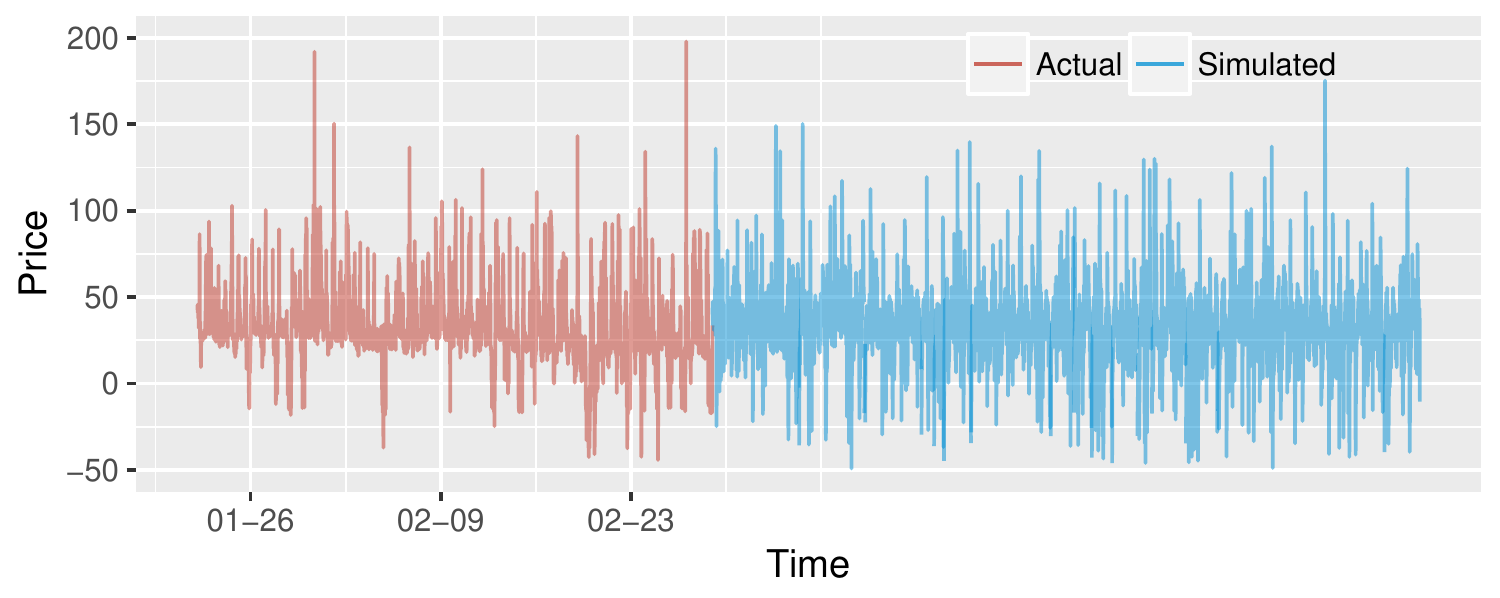}%
                \vspace{-5pt}

        \caption{Actual and Simulated Spot Prices}
        \label{fig:combined}
\end{figure}

Although we do not have access to EV charging reservation data, we examined a large dataset\footnote{Santa Monica Open Data, \texttt{https://data.smgov.net/Transportation/Parking-Lot-Counts/ng8m-khuz}} of parking space availability in public parking garages of Santa Monica, California (updated one every five minutes). An analysis of \emph{the changes in parking garage availability over time} across 14 parking garages and one year of five minute data from May 2016 to April 2017 (nearly 1.5 million records) provided us a estimated distribution for $\tau$. Observations of ``parking time'' were computed by first identifying times of peak utilization (i.e., the local maximums) and then analyzing the rate at which the garage empties. For example, if the garage empties quickly approximately one hour after the peak, then this suggests that many vehicles are parked for a length of one hour. The process yielded over 700,000 ``parking observations'' which were then used to estimate a distribution; see the second panel of Figure \ref{fig:taudist}. The first panel of Figure \ref{fig:taudist} shows a one month sample of the raw data from the parking utilization dataset. It is interesting to note that a significant portion of parking garage patrons park for around 1.25 hours.

\begin{figure}[h]
        \centering
        \begin{subfigure}[b]{0.48\textwidth}
                \centering
                \includegraphics[width=0.9\textwidth]{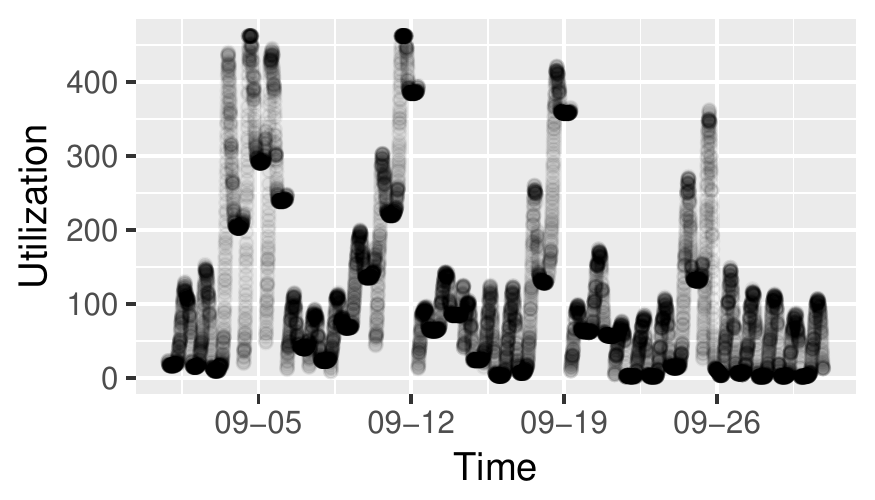}
                \caption{Parking Garage Utilization}
        \end{subfigure}
        \begin{subfigure}[b]{0.48\textwidth}
                \centering
                \includegraphics[width=0.9\textwidth]{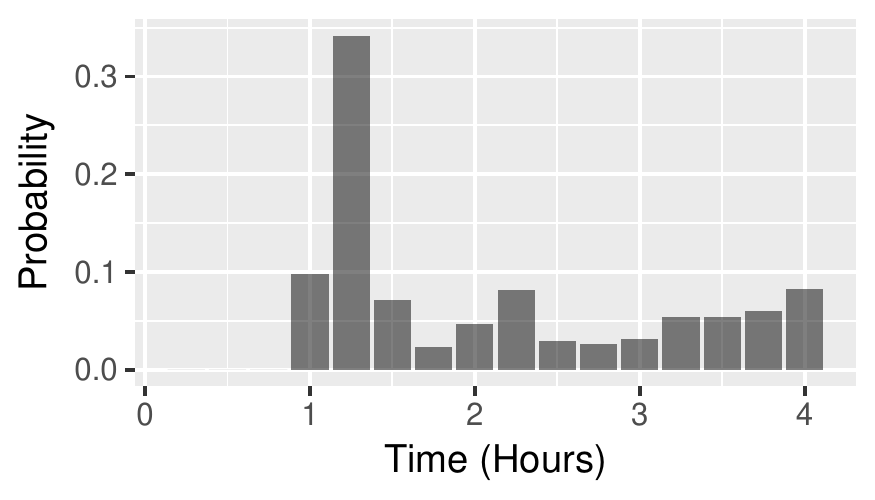}
                \caption{Distribution of $\tau$}
        \end{subfigure}
        \caption{Raw Parking Garage Data \& Estimated Distribution of Reservation Length $\tau$}
        \vspace{-10pt}
        \label{fig:taudist}
\end{figure}

The initial price is chosen to be $P_0=\$35/\text{MWh}$ (we follow the convention of wholesale electricity prices being written in units of MWh$^{-1}$ even though our setting is on the order of KWh) and the battery is assumed to be initially empty: $R_0=0$. The capacity of the battery is $\Rmax = 60$ kWh, in line with the entry version of the Tesla Model S from 2012 to 2015. This is a reasonable choice for a medium sized battery, as there are both EVs with smaller capacities (e.g., Nissan Leaf or BMW i3) and EVs with larger capacities (e.g., newer and higher end Tesla models). We set the charging fee to be $c_f=\$2.00/\text{hour}$. The parameters related to the inconvenience compensation are taken to be $p_\text{ref}= \$0.05/\text{kWh}$, $\gamma_h = 0.01$, and $\gamma_Y(y) = \log(1+\exp(y))$ (a smooth version of $\max(0,y)$), where $y$ is in units of \$/kWh.

\subsection{Exact MDP Results}
In order to properly solve the risk-averse MDPs, discretization is necessary. We use the method of \cite{Roy2003} to discretize $\psi_{t+1}$ to only contain integer outcomes and we ignore the outcomes that occur with probability less than $1.5 \cdot 10^{-3}$. Such a discretization gives us 260 distinct spot prices allows us to utilize the linear programming method for computing the CVaR, as pointed out in \cite{Rockafellar2000}. The resource state is also discretized so that $R_t$ takes values in $\{0, 1, \ldots, \Rmax \}$. Each MDP contains approximately 16{,}000 states per time period, leading to 270{,}000 total states when $\tau = 16$. 

We consider the risk-aversion set $\bar{\mathcal B} = \{\beta : (\lambda,\alpha)= \beta_0 = \beta_1 = \cdots =\beta_{T}\}$. To adequately explore $\bar{\mathcal B}$, we solve a total of 4{,}953 MDPs for parameter values of $\lambda \in \{0, 0.05, \ldots, 1\}$, $\alpha \in \{0.05, 0.1, \ldots, 0.95\}$, and $T \in \{4,5,\ldots,16\}$. Computing the solution to this set of MDPs requires a massive amount of computation and we do it only to illustrate the structural properties of the problem. This step is not necessary in practice, as we will see in Section \ref{sec:approx} that the approximation approach described in Section \ref{sec:select} is quite effective.

\begin{figure}[h]
        \centering
        \includegraphics[width=0.8\textwidth]{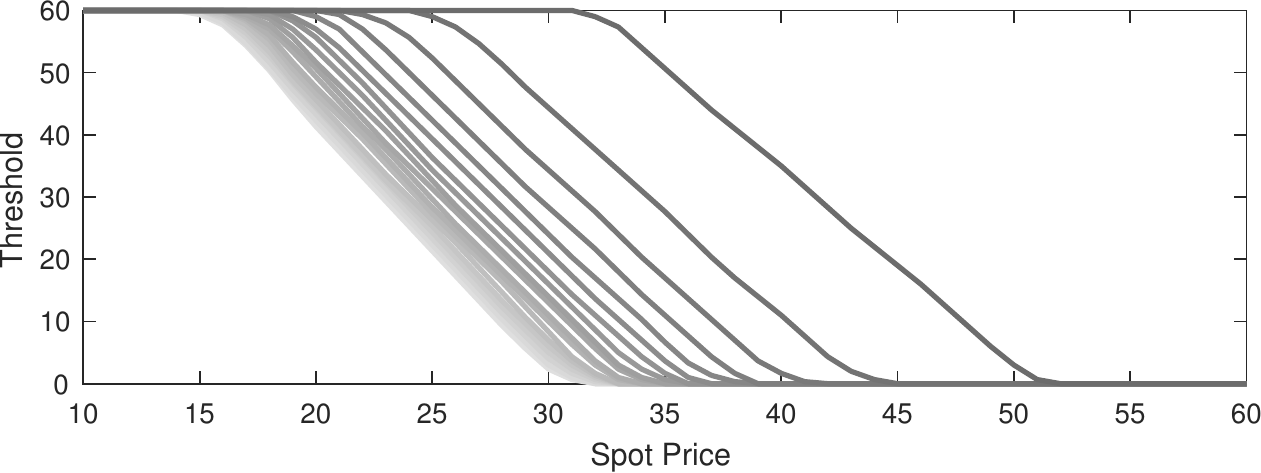}%
        \vspace{-5pt}
        \caption{Thresholds for $\lambda = 0.5$, $T=16$}
        \label{fig:threshold2d}
\end{figure}


We report our results using Figures \ref{fig:threshold2d} and \ref{fig:thresholds3d}. In Figure \ref{fig:threshold2d}, we plot the thresholds as a function of price for a fixed $\lambda = 0.5$. Each curve represents a different value of $\alpha$, which range from $0.05$ to $0.95$ (darker means higher $\alpha$ value). This serves to experimentally verify the statement of Theorem \ref{thm:threshspot}, that the thresholds are nonincreasing in $p$. Moreover, it verifies that the thresholds are nondecreasing in $\alpha$, as shown in the proof of Theorem \ref{thm:thmcompattwo}. We notice that at the for a given degree of risk-aversion, a \$30 change in spot price can alter the threshold by $\Rmax = 60$.
\begin{figure}[h]
        \centering
        \begin{subfigure}[b]{0.48\textwidth}
                \centering
                \includegraphics[width=\textwidth]{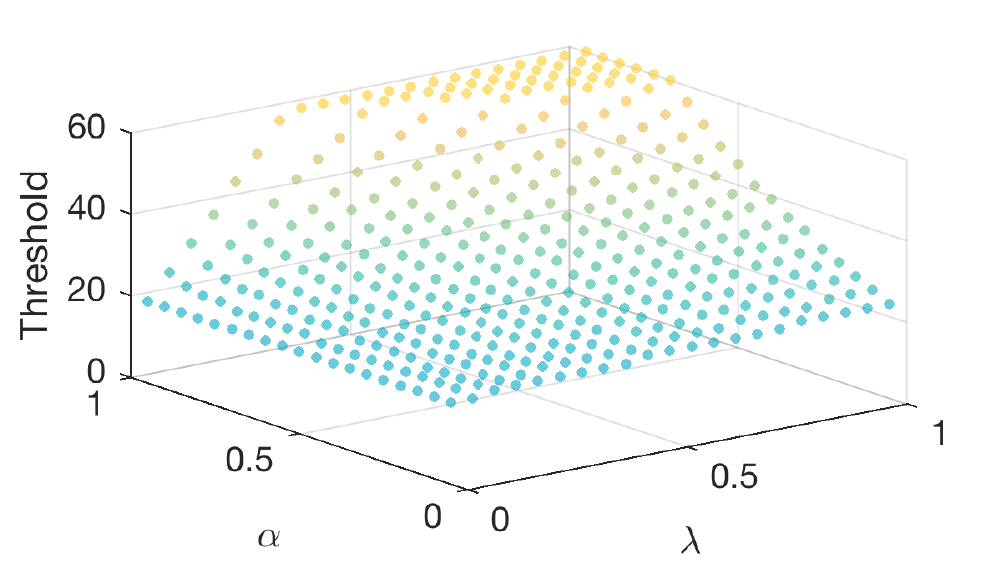}
                \caption{Thresholds for $p=25$, $t=16$}
                \label{subfig:threshold_a}
        \end{subfigure}
        \begin{subfigure}[b]{0.48\textwidth}
                \centering
                \includegraphics[width=\textwidth]{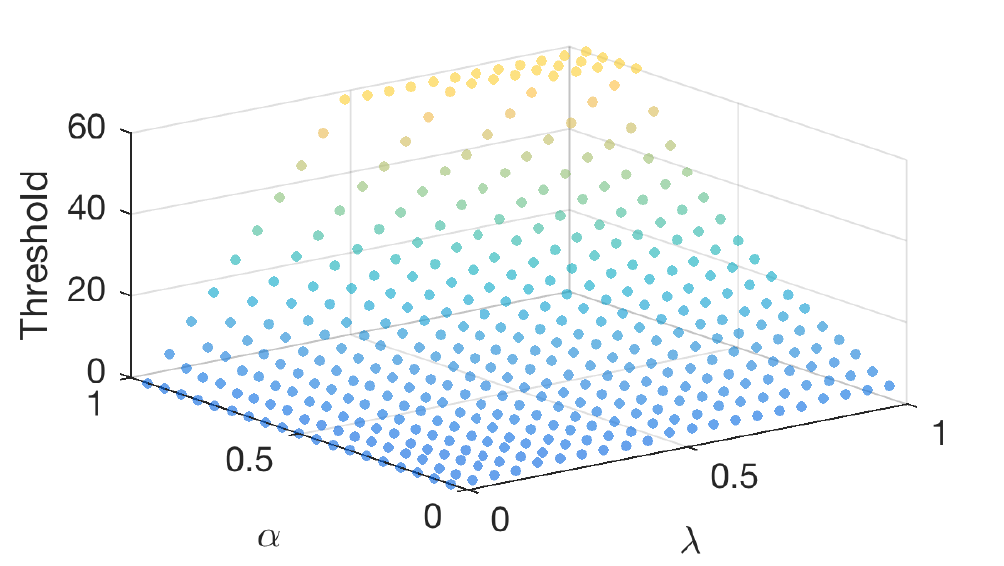}
                \caption{Thresholds for $p = 30$, $t=16$}
                \label{subfig:threshold_b}
        \end{subfigure}
        \caption{Behavior of Thresholds for Fixed $p$}
        \label{fig:thresholds3d}
\end{figure}

Figure \ref{fig:thresholds3d} fixes the spot price $p$ and displays the variation of the threshold as a function of $\beta$ (again, the thresholds are nondecreasing by Theorem \ref{thm:thmcompattwo}) for the MDP associated with $T=16$. We see that the degree of risk-aversion $(\lambda,\alpha)$ can dramatically change the policy: in the risk-neutral case for $p=25$, the threshold $r_{t,T}(p\,|\,\beta) \approx 20$, while in the most risk-averse case, $r_{t,T}(p\,|\,\beta) \approx 60 = \Rmax$ (meaning we should charge as much as possible). This observation emphasizes the importance of selecting an appropriate value for $\beta$, which we discuss in the next section.

\subsection{Selecting \texorpdfstring{$\beta$}{Beta} using Approximation}
\label{sec:approx}
One way to choose $\beta$ would be to compute and simulate a large batch of risk-averse policies corresponding to different values of $\beta$, similar to what is shown in Figure \ref{fig:thresholds3d} and the references discussed in Section \ref{sec:intro}. As discussed in Section \ref{sec:select}, this heuristic tuning method suffers from two major drawbacks: (1) it can be very computationally intense, due to the large number of $\beta$'s that needs to be explored, and (2) the final choice of $\beta$ necessarily comes from the pre-selected batch. For instance, a MATLAB implementation on two 14 core, 3.3 GHz Intel Haswell CPUs requires 2.5 minutes for each instance of $\beta$ (i.e., 13 MDPs with different horizons $T$). Solving these MDPs for 381 combinations of $\lambda$ and $\alpha$ thus requires nearly 16 hours of CPU time (or without parallelization, \emph{days} of CPU time).

We now illustrate the alternative approach that was outlined in Section \ref{sec:select}, wherein a small sample of $\beta$'s are chosen as ``observations'' that are fed into a structured regression model to estimate the practical reward and risk over the entire set $\bar{\mathcal B}$. We assume that the practical reward is given by (\ref{eq:practicalreward}), the negative of the total expected cost. Let $\delta = 0.3$ and suppose the practical risk is given by 
\[
\rho_\text{base}\,[f^-_\tau(\mathbf{S}^\pi_\tau)] = \mathbf{E} \bigl[ \mathbf{1}_{\{R_\tau / \Rmax \, \le \, (1-\delta)\}} \bigr] = \mathbf{P}(R_\tau / \Rmax \, \le \, (1-\delta)),
\]
the probability that the vehicle is less than $70\%$ charged.\footnote{We choose 70\% because a charge of around 80\% is usually considered an adequate charge. For example, Tesla states that ``charging above 80\% isn't typically necessary'' at  \texttt{http://www.tesla.com/supercharger}.} Let $\varepsilon \in [0,1]$ be the ``trade-off parameter.'' We consider the case where the decision maker aims to solve the practical problem (\ref{eq:tradeoffproblem})
by appealing to the framework described in Figure \ref{fig:practical} and searching over policies $\pi$ that optimize \emph{some instance} of the risk-averse MDP (similar to what is done in \cite{Maceira2015}). We proceed with the three steps outlined in Section \ref{sec:select}.

\vspace{7pt} \noindent \textbf{Step 1.} We let $\{\beta^1, \beta^2, \ldots, \beta^n\}$ be all combinations of $\lambda \in \{0,0.1, 0.2, \ldots, 1\}$ and $\alpha \in \{0.05, 0.15, \ldots, 0.95\}$. The corresponding MDPs are solved and the policies are simulated $N=100{,}000$ times to obtain estimates of the practical reward and risk. Note that we are using a coarser grid than before to save computation time, allowing the regression procedure to fill in the blanks.

\vspace{7pt} \noindent \textbf{Step 2.}  We elect to use a polynomial representation for $\hat{r}_\text{base}(\beta)$ and $\hat{\rho}_\text{base}(\beta)$, so that
\[
\hat{r}_\text{base}(\beta) = \mathbf{w}_{\Gamma}^\mathsf{T} \, \phi(\lambda, \alpha) \quad \text{and} \quad  \hat{\rho}_\text{base}(\beta) = \mathbf{w}_{\Psi}^\mathsf{T} \, \phi(\lambda, \alpha),
\]
where $\mathbf{w}_{\Gamma}$ and $\mathbf{w}_{\Psi}$ are the weight vectors and $\phi(\lambda, \alpha)$ is the vector of all degree 10 monomials (66 terms).

\begin{figure}[h]
        \centering
        \begin{subfigure}[b]{0.48\textwidth}
                \centering
                \includegraphics[width=\textwidth]{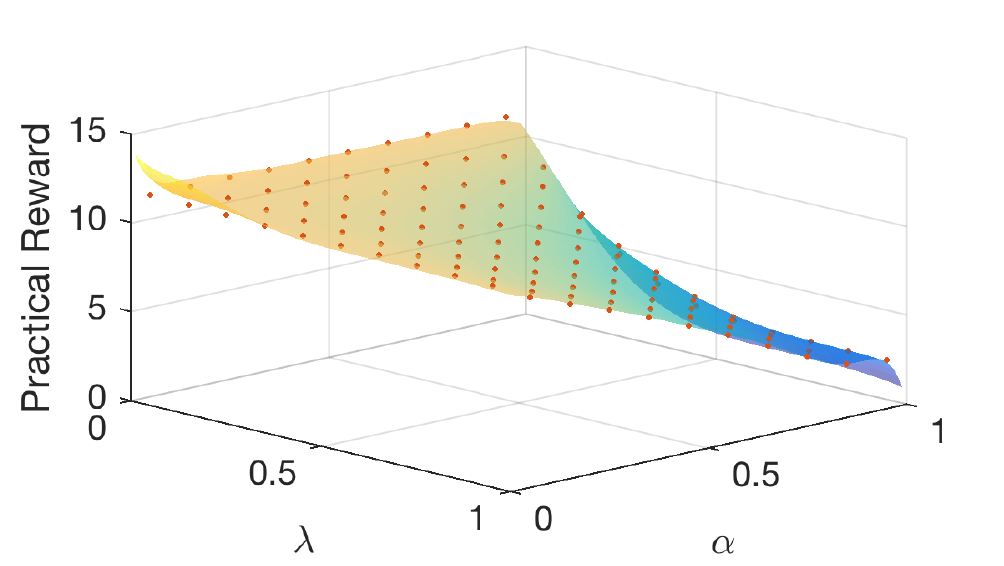}
                \caption{Approximated Practical Reward $\hat{r}_\text{base}$} 
                \label{subfig:threshold_a}
        \end{subfigure}
        \begin{subfigure}[b]{0.48\textwidth}
                \centering
                \includegraphics[width=\textwidth]{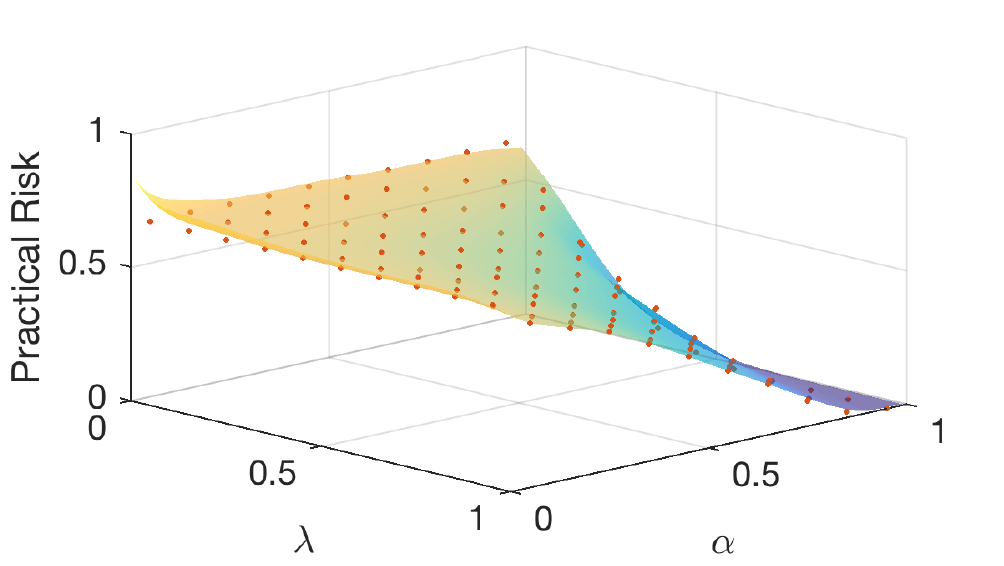}
                \caption{Approximated Practical Risk $\hat{\rho}_\text{base}$}
                \label{subfig:threshold_b}
        \end{subfigure}
        \caption{Practical Reward and Risk Metrics using Regression}
        \label{fig:approxprac}
\end{figure}

Using the SPOT toolbox \citep{Megretski2013} and the MOSEK solver, we carried out an $l_1$ regression procedure to compute the coefficients $\mathbf{w}_{\Gamma}, \mathbf{w}_{\Psi}$ --- as mentioned in Section \ref{sec:select}, sum-of-squares constraints \citep{Ahmadi2014,Ahmadi2015} are used to compute $\mathbf{w}_{\Psi}$ to enforce the compatibility result of Theorem \ref{thm:thmcompattwo}. The regression procedure requires roughly four seconds of CPU time (2 seconds for the reward and 2 seconds for the risk).

Figure \ref{fig:approxprac} shows our estimated practical metric surface over $\beta \in \bar{ \mathcal B}$ along with the metrics associated with the sample set of MDPs as red dots. In particular, we point out that the policy obtained by solving the \emph{risk-neutral} (RN) MDP, i.e., when $\lambda = 0$, generates a profit of $\$11.30$, but leaves the EV inadequately charged 65.5\% of the time.\footnote{One may point out that surely we could change the inconvenience compensation scheme to mitigate this behavior. However, the idea of risk-averse optimization is to take the ``physical characteristics'' of the problem \emph{as given} (i.e., fix the current compensation scheme) and produce risk-averse policies.} This level of risk would not be acceptable in practical settings.

\vspace{7pt} \noindent \textbf{Step 3.} In the final step of the procedure, we use the approximations $\mathbf{w}_{\Gamma}, \mathbf{w}_{\Psi}$ to select a choice of $\beta$. Specifically, we solve the problem
\[
\displaystyle \text{maximize}_{\beta \in\, [0,1] \times (0,1)} \; \; \hat{r}_\text{base}(\beta) \; \; \text{subject to} \; \; \hat{\rho}_\text{base}(\beta) \le \varepsilon,
\]
where the maximizer is denoted $\hat{\beta}(\varepsilon) = (\hat{\lambda}(\varepsilon), \hat{\alpha}(\varepsilon))$. This can be thought of as a ``recommended degree of risk-aversion'' for the particular level of $\varepsilon$. The policy selected for eventual implementation is $\pi^{\hat{\beta}(\varepsilon)}_{\tau}$ --- the optimal policy with respect to the risk-averse MDP specified using $\hat{\beta}(\varepsilon)$. Of course, the solution $\hat{\beta}(\varepsilon)$ is generally not one of the sampled points, i.e.,  $\hat{\beta}(\varepsilon) \not \in \{\beta^1 ,\beta^2, \ldots, \beta^n\}$, so the final step before implementation is to solve the MDPs associated with $\hat{\beta}(\varepsilon)$. The recommended $\hat{\beta}(\varepsilon)$ and the practical reward and risk of the associated optimal policies for ten reasonable values for $\varepsilon$, ranging from $0.01$ to $0.10$, are shown in Table \ref{table:results}. The table also provides a comparison to the default policy (charge continuously) and the optimal risk-neutral (RN) solution, which we have given in the first and last rows, respectively.

\renewcommand{\arraystretch}{1.2}
\begin{table}[h]
\centering
\small
\begin{tabular}{@{}cccccccc@{}}\toprule
 \quad Trade-off $\varepsilon$  \quad &  \quad $\hat{\lambda}(\varepsilon)$  \quad &  \quad $\hat{\alpha}(\varepsilon)$  \quad &  \quad  Reward (\% of RN) \quad &  \quad  Risk (\% of RN)  \quad \\
 \midrule
Default & -    & -     & \$1.90 (16.8\%) & 0.000 (0.0\%)\\
0.01  & 0.937 & 0.914 & \$2.74 (24.3\%) & 0.002 (0.3\%)\\
0.02  & 0.876 & 0.901 & \$2.85 (25.3\%) & 0.009 (1.3\%)\\
0.03  & 0.827 & 0.896 & \$2.97 (26.3\%)  & 0.019 (2.9\%)\\
0.04  & 0.793 & 0.891 & \$3.11 (27.5\%) & 0.033 (5.0\%)\\
0.05  & 0.761 & 0.888 & \$3.25 (28.8\%) & 0.041 (6.2\%)\\
0.06  & 0.747 & 0.881 & \$3.39 (30.0\%) & 0.054 (8.2\%)\\
0.07  & 0.715 & 0.881 & \$3.52 (31.2\%) & 0.068 (10.4\%)\\
0.08  & 0.681 & 0.883 & \$3.66 (32.4\%) & 0.072 (10.9\%)\\
0.09  & 0.669 & 0.878 & \$3.80 (33.7\%) & 0.090 (13.7\%)\\
0.10 & 0.61  & 0.89  & \$3.94  (34.9\%) & 0.095 (14.5\%)\\
RN & - & - & \$11.30 (100.0\%) & 0.655 (100.0\%)\\
\bottomrule
\end{tabular}
\caption{Practical Reward and Risk of the Recommended Risk-Aversion Levels}
\label{table:results}
\end{table}

Next, we show in Figure \ref{fig:betapath} the path of $\hat{\beta}(\varepsilon) = (\hat{\lambda}(\varepsilon), \hat{\alpha}(\varepsilon))$ as $\varepsilon$ varies from 0.05 up to 0.6 (since we know that the RN solution incurs a practical risk of 0.655). Note that this figure is purely for methodological interest, as it is unlikely that $\varepsilon$ values greater than 0.10 would be useful in practice. Although approximations from the regression are used to create the path of $\beta$, Figure \ref{fig:betapath} suggests that the relationship between (\ref{eq:tradeoffproblem}) and the risk-averse MDP policies is complex and that there may not a structure that is simple to discern. Additional research in this area is needed.

\begin{figure}[h]
        \centering
        \includegraphics[width=0.8\textwidth]{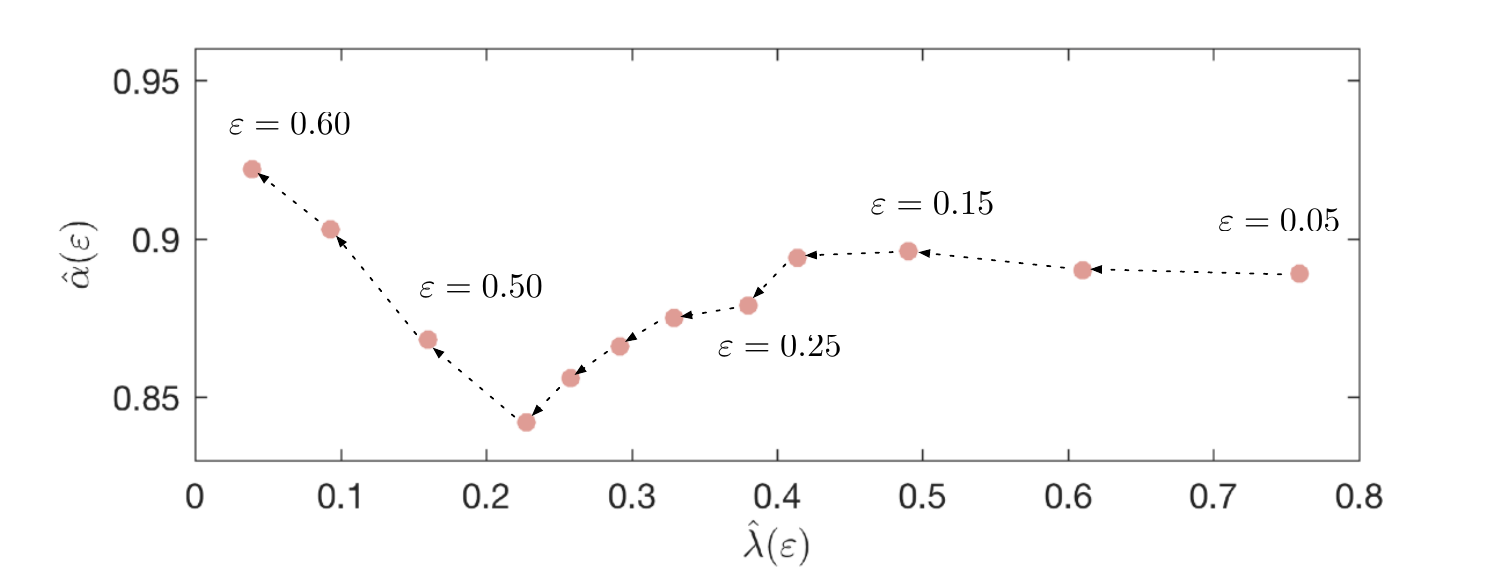}%
        \caption{The Path of $\hat{\beta}$}
                \vspace{-20pt}
        \label{fig:betapath}
\end{figure}

\section{Insights and Discussion}
Our problem focuses on dynamic EV charging in the presence of a reservation system. However, ``dynamic EV charging'' may occur under various other related settings and system designs, and we believe that the results from our study apply there as well. Examples of these related problems are given below, along with a commentary on how our model can be extended to these settings.
\begin{itemize}
\item[--] In a first-come, first-served EV charging station without reservations, the control system may attempt to ``guess'' at the return time of the customer and attempt to delay charging to reduce costs. In this model, the distribution of $\tau$ (the customer arrival time) is incorporated into the MDP, while in our current model, the distribution of $\tau$ defines multiple instances of the MDP.
\item[--] When a charging station has multiple ports that share the same power source (see, e.g., ChargePoint Express Plus and ChargePoint Power Cube), it is necessary to decrease the power output when additional vehicles join, creating the possibility that vehicles are not charged up to the customers' expectations. Our model can be extended to this case by specifying a pre-determined rule for how power capacity should be split between vehicles; for example, in the interest of fairness, the rule could be that the power is always split evenly. Our model could then be applied independently at each port with the addition that the charging capacity $\xmax$ is now a random variable (that decreases whenever additional vehicles are plugged into the station and increases when vehicles leave).
\item[--] Vehicle-to-grid systems take advantage of the storage capacity of idle electric vehicles, allowing them to provide ancillary services to the grid in times of need \citep{Rotering2011}. For our model to handle this case, we would allow the charging decision $x_t$ to take both postive and negative values (negative values represent the delivery of energy to the grid).
  \end{itemize}
A common thread of the above is that the customer is potentially disappointed when the charge level is inadequate at the time of use. Compensation for the inconvenience, in whatever form, monetary or otherwise, should then be provided. In this paper, we show an example where, under a reasonable compensation scheme, the optimal risk-neutral policy to the dynamic EV charging problem leaves the vehicle inadequately charged nearly \emph{two-thirds of the time}. This suggests that for the broader class of dynamic EV charging problems, simply optimizing with respect to the expected total cost under any given inconvenience compensation scheme may produce undesirable policies. 

One may point out that a trade-off between charging costs and ``risk'' (in the form of the compensation) is already being made in the \emph{risk-neutral problem}. This is true to an extent; the issue is that trade-off induced by these competing costs is, in some sense, misaligned, as the firm cannot implement a policy that provides an inadequate charge 65.5\% of the time. Thus, the question is: if we take the physical characteristics of the problem as fixed (i.e., we do not re-write the compensation scheme), then how can we produce a policy that is more risk-averse? The introduction of risk measures provides an additional lever to ``tip the scales'' when the cost function itself provides an incorrect trade-off (utility functions are another approach). By implementing our risk-averse MDP/practical metrics approach on our reservation-based EV charging model, we are able to discover policies with much more reasonable trade-offs (see Table \ref{table:results}).

Most EV charging stations today use the default continuous charging policy, so it is perhaps more informative for our results to be viewed from the point of view of this standard practice. We see from Table \ref{table:results} that if the manager of the charging station is willing to incur a small risk of $0.002$, then a dynamic charging policy can increase profits by 44\% (from \$1.90 to \$2.74). This suggests that EV charging stations can be operated at a significantly higher level of profitability if dynamic charging is adopted.

\section{Conclusion}
In this paper, we study EV charging in a setting where one attempts to exploit volatility in the spot market by charging only when the prices are low. At the same time, the risk of unsatisfied customers who may return to undercharged vehicles must be controlled. The problem is specified as a risk-averse MDP formulated using a dynamic (nested) risk measure, but we point out that the objective function has no practical interpretation. Our main methodological contribution lies in formalizing a procedure wherein risk-averse MDPs are used as a computational tool to produce policies that are evaluated using ``practical'' reward and reward metrics (termed the \emph{base model}). We define a notion of \emph{risk compatibility} and prove that under certain conditions, additional risk-aversion in the dynamic risk measure sense leads to additional risk-aversion in the practical sense as well. Next, we propose an approximation procedure that allows a decision-maker to easily search within a set of risk-averse MDPs for a ``practical policy'' using a regression model on a few solved instances. Finally, we provide numerical results based on real data that support the effectiveness of the approximation procedure. The numerical experiments suggest that dynamic EV charging can significantly improve the profitability of EV charging stations when a small amount of risk is incurred.

\clearpage
\appendix
\section{Proofs}
\label{sec:appendix}

\propconv*
\begin{proof}
First note that $V_{t,T}(s\,|\,\beta)$ is trivially convex in $r$. We let the induction hypothesis be that $V_{t+1,T}( \, \cdot \, | \, \beta)(s)$ is strictly convex in $r$, from which we can show (i), (ii), and (iii) for $t$.
For a feasible decision $x \in \mathcal X(r)$, let
\begin{align*}
Q_{t,T}(s,x\,|\,\beta) = xp - c_f + \tilde{V}_{t,T}(r+x,p\, | \, \beta)
\end{align*}
and $\mathcal Q = \{(r,x): x \in \mathcal X(r), r\in \mathcal R\}$. From the convexity of the mapping $(r,x) \mapsto V_{t+1,T}( r + x, p \, e^{-\kappa_Y} + \psi_{t+1}\,|\,\beta)$ on $\mathcal Q$ for each realization of $\psi_{t+1}$ and the convexity (in the space of random variables) of $\rhooperator_{\beta_t}$, we obtain that
\[
(r,x) \mapsto \tilde{V}_{t,T}(r+x,p\, | \, \beta) = \rhooperator_{\beta_t} \bigl[V_{t+1,T}( r + x, p \, e^{-\kappa_Y} + \psi_{t+1}\,|\,\beta)  \bigr]
\]
is convex on $\mathcal Q$. This implies property (ii) and it follows that $Q_{t,T}(s,x\,|\,\beta)$ is convex in $(r,x)$ on $\mathcal Q$, which verifies (i).

Since we can write $V_{t,T}(s\,|\,\beta) = \min_{x \in \mathcal X(r)} Q_{t,T}(s,x\,|\,\beta)$ and $\mathcal X(r)$ is nonempty, the well-known property that convexity is preserved under minimization can be applied; by \cite[Proposition B-4]{Heyman2003}, $V_{t,T}(s\,|\,\beta)$ is convex in $r$. This proves property (iii) and we have completed the inductive step and also verified (i) and (ii) along the way.
\end{proof}

\propmonop*
\begin{proof}
We first focus on monotonicity in $p$ and proceed via induction. For each resource state $r \in \mathcal R \setminus \{\Rmax\}$, the function $V_{T,T}(s \, | \, \beta)$ is increasing in the spot price $p$ by Assumption \ref{ass:compensationlip}. This completes the base case. We assume the same is true for $V_{t+1,T}(r,p\,|\,\beta)$ and complete the inductive step by proving that the property holds for $V_{t,T}(r,p \, | \, \beta)$. To do so, we consider 
\[
\tilde{V}_{t,T}(r+x,p\, | \, \beta) = (1-\lambda_t) \, \mathbf{E} \bigl[V_{t+1,T}( r+x, p \, e^{-\kappa_Y} + \psi_{t+1}\,|\,\beta) \bigr] + \lambda_t \, c_{t,T}(r+x,p\,|\,\beta).
\]
Since $\lambda$ can take values of $0$ or $1$, we analyze both terms. By the induction hypothesis, the terms within the expectation are increasing in $p$ for every realization of $\psi_{t+1}$, so the first term is increasing in $p$.

Now, we consider $c_{t,T}(r+x,p\,|\,\beta)$. Let $p' > p$ so that by the induction hypothesis, we know
there exists a deterministic $\epsilon' > 0$ such that
\[
V_{t+1,T}( r+x, p \, e^{-\kappa_Y} + \psi_{t+1}\,|\,\beta) + \epsilon' \le V_{t+1,T}( r+x, p' \, e^{-\kappa_Y} + \psi_{t+1}\,|\,\beta) \quad a.s.
\]
Hence, by the translation invariance and monotonicity axioms of coherent risk measures \cite{Artzner1999}, we have $c_{t,T}(r+x,p\,|\,\beta) + \epsilon' \le c_{t,T}(r+x,p'\,|\,\beta)$. We have shown that $\tilde{V}_{t,T}(r+x,p\, | \, \beta)$ is increasing. Thus, $xp - c_f + \tilde{V}_{t,T}(r+x,p\, | \, \beta)$ is increasing in $p$ for every $x \in \mathcal X(r)$, implying that $V_{t,T}(\,\cdot\,|\,\beta)$ is too.
\end{proof}

\thmpolicy*
\begin{proof}
To obtain an equivalent formulation of (\ref{eq:postdecbellman}) in terms of the post-decision resource state, we make a substitution $\tilde{r} = r+x$ and $x=\tilde{r}-r$. Define $\mathcal R(r) = \{\tilde{r} : r \le \tilde{r} \le \min\{r+\xmax, \Rmax\}\}$ as the set of possible post-decision resource levels. We have
\begin{align}
x_{t,T}(r,p \, | \, \beta) \in \Bigl(\argmin_{\tilde{r} \in \mathcal R (r)} \; \tilde{r}\,p + \tilde{V}_{t,T}(\tilde{r},p \, | \, \beta)\Bigr) -r.
\label{eq:optpolicytilde}
\end{align}
Note that the objective is convex in $\tilde{r}$ by Proposition \ref{prop:conv}.
Since $\mathcal R(r) \subseteq \mathcal R$, consider the solution $r_t(p\,|\,\beta)$ to the relaxed optimization problem
\begin{equation}
r_{t,T}(p\,|\,\beta) \in \argmin_{\tilde{r} \in \mathcal R} \; \tilde{r} \, p + \tilde{V}_{t,T}(\tilde{r},p \, | \, \beta).
\label{eq:rt}
\end{equation}
Clearly, if $r_{t,T}(p\,|\,\beta) \in \mathcal R(r)$, then it is an optimal solution to the optimization problem within (\ref{eq:optpolicytilde}) and the optimal decision is $r_{t,T}(p\,|\,\beta) - r$. There are two ways in which $r_{t,T}(p\,|\,\beta) \not \in \mathcal R(r)$ can happen. First, if $r_{t,T}(p\,|\,\beta) < r$, then by convexity, $\tilde{r} \, p+ \tilde{V}_{t,T}(\tilde{r},p \, | \, \beta)$ is nondecreasing on $\mathcal R(r)$ and an optimal solution is $\tilde{r} = r$, implying an optimal decision of $0$. The second case is when $r+\xmax < r_{t,T}(p\,|\,\beta) \le \Rmax$. Again, by convexity, it must be true that $\tilde{r} \, p + \tilde{V}_{t,T}(\tilde{r},p \, | \, \beta)$ is nonincreasing on $\mathcal R(r)$ and an optimal solution is $\tilde{r} = r+\xmax$. This necessarily gives us an optimal decision of $\xmax$. Finally, combining these three cases allows us to conclude the structure of the optimal policy stated in the theorem. 
\end{proof}

\lemlip*
\begin{proof}
First, the terminal cost function
\[
V_{t,T}(r,p\,|\,\beta) = \bigl[ 1+\gamma_h\, h_t(r) + \rhooperator_{\beta_T}  [ \gamma_Y(p \, e^{-\kappa_Y} + \psi_{T+1,Y}) ] \bigr] \, h_t(r) \, p_\text{ref}
\]
is clearly Lipschitz in $p$ for each $r$ by Assumption \ref{ass:compensationlip} and the inequality $|\rhooperator_{\beta_t}(A) - \rhooperator_{\beta_t}(B)| \le \rhooperator_{\beta_t} (|A-B| )$.  For each fixed $r$, we denote the Lipschitz parameter by $L_{V_{T,T}|\,\beta}(r)$ and then let $L_{V_{T,T}|\,\beta} \!\! = \max_{r \in \mathcal R} L_{V_{T,T}|\,\beta}(r)$. We now proceed by backward induction on $t$. Using the inequality $|\min f - \min g | \le \max |f-g|$, we have
\begin{equation}
\bigl|V_{t,T}(r,p\,|\,\beta) - V_{t,T}(r,p'\,|\,\beta) \bigr| \le \max_{x \in \mathcal X(r)} \, \bigl |  \tilde{V}_{t,T}( r+x,p\,|\,\beta) -  \tilde{V}_{t,T}( r+x,p'\,|\,\beta)\bigr|.
\label{eq:lipschitzp}
\end{equation}
Using the definition of the post-decision value function and (\ref{eq:psidef}), we have
\[
\tilde{V}_{t,T}( r+x,p\,|\,\beta) =   \rhooperator_{\beta_t} \bigl[V_{t+1,T}( r+x, p \, e^{-\kappa_Y} + \psi_{t+1}\,|\,\beta)\bigr].
\]
Again using the inequality $|\rhooperator_{\beta_t}(A) - \rhooperator_{\beta_t}(B)| \le \rhooperator_{\beta_t} (|A-B| )$ along with (\ref{eq:lipschitzp}), we conclude that
\[
\bigl|V_{t,T}(r+x,p\,|\,\beta) - V_{t,T}(r+x,p'\,|\,\beta) \bigr| \le e^{-\kappa_Y}  L_{V_{t+1,T}|\,\beta} \, |p - p'|
\]
for any $r+x$. Thus, we can set $L_{V_{t,T}|\,\beta} = e^{-\kappa_Y} L_{V_{t+1,T}|\,\beta}$ to complete the proof.
\end{proof}

\lemdiffvar*
\begin{proof}
We apply the simple property that value at risk is invariant under a monotone transformation in the following sense. Consider a random variable $X$ and an increasing function $l$. Then, since $\P[X \le \text{VaR}_{\alpha_t}(X)] = \P[l(X) \le l(\text{VaR}_{\alpha_t}(X))] = \alpha$, it follows that $\text{VaR}_{\alpha_t}(l(X)) = l(\text{VaR}_{\alpha_t}(X))$.
By the monotonicity property of Proposition \ref{prop:monop}, we have 
\[
v_{t,T}(r,p \, | \, \beta) = \text{VaR}_{\alpha_t}[V_{t+1,T}(r,p\,e^{-\kappa_Y} + \psi_{t+1}\,|\,\beta)] = V_{t+1,T}(r,p\,e^{-\kappa_Y} + \text{VaR}_{\alpha_t}(\psi_{t+1})\,|\,\beta)
\]
We conclude by noting that the Lipschitz property of Lemma \ref{lem:lip} implies $V_{t+1,T}(r, \,\cdot\,|\,\beta)$ is differentiable almost everywhere.
\end{proof}

\lemsensitivity*
\begin{proof}
The first part of the equation is simply an interchange of differentiation and expectation:
\[
\partial_r \mathbf{E}\bigl[ V_{t+1}^\beta(r,p \, e^{-\kappa_Y} + \psi_{t+1}) \bigr] = \mathbf{E}\bigl[ \partial_r V_{t+1}^\beta(r,p \, e^{-\kappa_Y} + \psi_{t+1}) \bigr].
\]
By Proposition \ref{prop:conv}, $V_{t+1}^\beta(r,p \, e^{-\kappa_Y} + \psi_{t+1})$ is Lipschitz on $[0,\Rmax]$ and thus the interchange follows by the dominated convergence theorem. To show the second part, we first apply \cite[Theorem 3.1]{Hong2009} to obtain
\begin{equation}
\begin{aligned}
&\partial_r \textnormal{CVaR}_{\alpha_t}[V_{t+1,T}(r,p \, e^{-\kappa_Y} + \psi_{t+1}\,|\,\beta)] \, |_{r=r_0} \\
&=   \mathbf{E}\bigl[ \partial_r V_{t+1,T}(r,p \, e^{-\kappa_Y} + \psi_{t+1} \, | \, \beta) \, | \, V_{t+1,T}(r,p \, e^{-\kappa_Y} + \psi_{t+1} \, | \, \beta) \ge v_{t,\alpha}(r,p\,|\,\beta)  \bigr] \, |_{r=r_0.}
\end{aligned}
\label{eq:dcvar}
\end{equation}
Invoking this theorem requires verifying \cite[Assumptions 1-3]{Hong2009}. \cite[Assumptions 1-2]{Hong2009} are verified immediately by Proposition \ref{prop:conv} and Lemma \ref{lem:diffvar}. \cite[Assumption 3]{Hong2009} states that $\{V_{t+1,T}(r,p \, e^{-\kappa_Y}  + \psi_{t+1}\,|\,\beta) = v_{t,\alpha}(r,p\,|\,\beta)\}$ is a zero probability event, which is clear by Proposition \ref{prop:monop}.
To finish the proof, we note that the conditioning event $\{V_{t+1,T}(r,p \, e^{-\kappa_Y}  + \psi_{t+1}\,|\,\beta) \ge v_{t,\alpha}(r,p\,|\,\beta)\}$ is equivalent to
\[
\{ V_{t+1,T}(r,p\,e^{-\kappa_Y} + \psi_{t+1}\,|\,\beta) \ge V_{t+1,T}(r,p\,e^{-\kappa_Y} + \text{VaR}_{\alpha_t}(\psi_{t+1})\,|\,\beta)\}
\]
by monotonicity in the spot price (Proposition \ref{prop:monop}). Inverting both sides gives the desired characterization $\{ \psi_{t+1} \ge \text{VaR}_{\alpha_t}(\psi_{t+1})\}$. 
\end{proof}

\thmthreshspot*
\begin{proof}
The proof is by backward induction on $t$. The \emph{induction hypothesis} is:
\begin{itemize}
\item[--] The quantity $p + \partial_r \tilde{V}_{t,T}(r,p\,|\,\beta)$ is nondecreasing.
\end{itemize}
Due to the characterization of $r_{t,T}(p\,|\,\beta)$ as the solution to the minimization problem $\tilde{r}\, p + \tilde{V}_{t,T}(\tilde{r},p\,|\,\beta)$ over $\tilde{r} \in \mathcal R$, this induction hypothesis implies that $r_{t,T}(p\,|\,\beta)$ is nonincreasing in $p$.

\vspace{7pt} \noindent \textbf{Some Remarks.} Before continuing with the proof, we make some remarks regarding differentiability in $r$ and $p$. The value function $V_{t+1,T}(r,\, \cdot \,|\,\beta)$ is \emph{piecewise continuously differentiable}\footnote{The notion of \emph{piecewise continuously differentiability} for some function $f : [a,b] \rightarrow \mathbb R$ states that there exists $a_i$ with $a=a_0 < a_1 < \cdots < a_n=b$ such that $f$ is continuously differentiable when restricted to $[a_i,a_{i+1}]$ for each $i < n$. More precisely, for a fixed $i$, $f'(x)$ at $x \in (a_i, a_{i+1})$ is the derivative, while $f'(a_i)$ is the right derivative, and $f'(a_{i+1})$ is the left derivative and $f'$ is continuous on $[a_i,a_{i+1}]$. By Theorem \ref{thm:optimal_policy}, the optimal policy $x_{t,T}(r,p\,|\,\beta)$ is piecewise linear and hence, piecewise continuously differentiable in $r$.} in the resource state $r$. This can be shown via an induction argument combined with applications of Lemma \ref{lem:sensitivity}, the envelope theorem, and the fact that $x_{t,T}(r,p\,|\,\beta)$ is piecewise linear. Thus, our argument centers on resource states in the interval $(0,\Rmax)$ at which differentiability holds. We then extend the result of the theorem to directional derivatives at the nondifferentiable endpoints/breakpoints by invoking piecewise continuous differentiability and taking limits for each ``piece.'' 

\vspace{7pt} \noindent  \textbf{Base Case.} We first consider the base case, $t+1=T$, i.e., the case of the terminal value function $V_{t,T}(s\,|\,\beta)$.
By Assumption \ref{ass:compensationlip}, the derivative
\begin{align*}
\partial_r V_{t+1,T}(r,p \, &e^{-\kappa_Y} + \psi_{t+1}\,|\,\beta)\\
&= -p_\text{ref} \bigl[ 1 + 2 \gamma_h \, h(r) +  \rhooperator_{\beta_T} [\gamma_Y((p \, e^{-\kappa_Y} + \psi_{t+1}) \, e^{-\kappa_Y} + \psi_{t+1,Y})] \bigr] 
\end{align*}
is a decreasing function in $p$, and invoking the inequality $|\rhooperator_{\beta_t}(A) - \rhooperator_{\beta_t}(B)| \le \rhooperator_{\beta_t} (|A-B| )$, we have that for any $p$ and $q$,
\begin{align*}
\bigl|\partial_r V_{t+1,T}(r,p \, e^{-\kappa_Y} + \psi_{t+1}\,|\,\beta) - \partial_r V_{t+1,T}(r,q \, e^{-\kappa_Y} + \psi_{t+1}\,|\,\beta)\bigr| \le p_\text{ref} \, L_{\gamma_Y} e^{-2\kappa_Y}\, |p  - q|.
\end{align*}
Since the Lipschitz constant satisfies $L_{\gamma_Y} \le p_\text{ref}^{-1} \, e^{2\kappa_Y}$ (by Assumption \ref{ass:compensationlip}), the right-hand-side of the inequality above can be bounded by $|p-q|$. By Lemma \ref{lem:sensitivity}, we have
\begin{align*}
\partial_r \tilde{V}_{t,T}(r,p\,|\,\beta) = (1&-\lambda_t) \, \mathbf{E}\bigl[ \partial_r V_{t+1,T}(r,p \, e^{-\kappa_Y} + \psi_{t+1}\,|\,\beta) \bigr] \\
& + \lambda_t \, \mathbf{E}\bigl[ \partial_r V_{t+1,T}(r,p \, e^{-\kappa_Y} + \psi_{t+1}\,|\,\beta) \, | \, \psi_{t+1} \ge \textnormal{VaR}_{\alpha_t}(\psi_{t+1})  \bigr] ,
\end{align*}
which satisfies $| \partial_r \tilde{V}_{t,T}(r,p\,|\,\beta) - \partial_r \tilde{V}_{t,T}(r,q\,|\,\beta)  \bigr| \le |p-q|$. Therefore, it follows that the quantity $p + \partial_r \tilde{V}_{t,T}(r,p\,|\,\beta)$ is nondecreasing in $p$.

\vspace{7pt} \noindent  \textbf{Inductive Step (Part 1).} 
Let us now consider the value function at time $t$ evaluated at the spot price $P = p \, e^{-\kappa_Y} + \psi_{t}$ (i.e., from the perspective of time $t-1$). Recall that the optimal policy at time $t$, by Theorem \ref{thm:optimal_policy}, is to charge up to (or as close as possible to) the threshold $r_{t,T}(p\,|\,\beta)$. We define the price sets $\mathscr P_\text{low}(r) = \{P: r+\xmax < r_{t,T}(p\,|\,\beta)\}$, $ \mathscr P_\text{med}(r) = \{P: r \le r_{t,T}(p\,|\,\beta) \le r+\xmax\}$, and  $\mathscr P_\text{high}(r) = \{P:r_{t,T}(p\,|\,\beta) < r\}$, where it is possible for some of these sets to be empty. 

Due to $r_{t,T}(p\,|\,\beta)$ being nonincreasing in $p$, these price sets are disjoint and monotone in the sense that any element of $\mathscr P_\text{low}(r)$ is no greater than any element of $\mathscr P_\text{med}(r)$, and any element of $\mathscr P_\text{med}(r)$ is no greater than any element of $\mathscr P_\text{high}(r)$. See Figure \ref{fig:pricesets} for an illustration.
\begin{figure}[h]
        \centering
        \includegraphics[width=\textwidth]{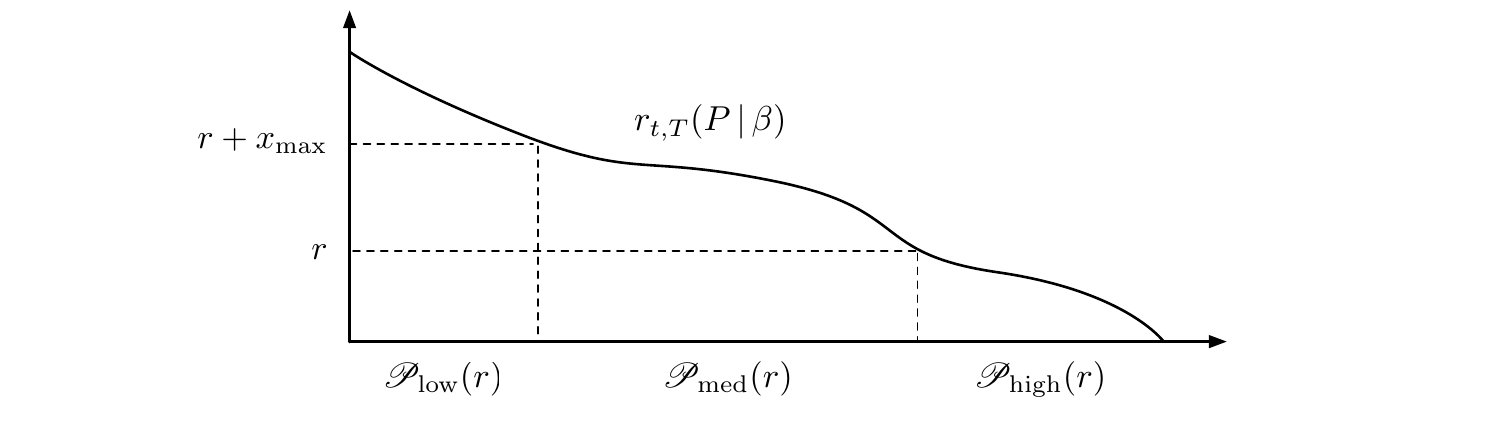}
        \caption{Illustration of the Price Sets}
        \label{fig:pricesets}
\end{figure}

The value function can be written using the optimal policy and the post-decision value function:
\[
V_{t,T}(r,p\,|\,\beta) = x_{t,T}(r,p \, | \, \beta) \cdot P - c_f + \tilde{V}^\beta_{t,T}(r+x_{t,T}(r,p \, | \, \beta) ,P).
\]
Doing casework on each of the spot price intervals, we have
\begin{equation*}
V_{t,T}(r,p\,|\,\beta) = \begin{cases}
\xmax \cdot P -c_f + \tilde{V}_{t,T}(r+\xmax,P\,|\,\beta) &\text{ if } P \in \mathscr P_\text{low}(r),\\
(r_{t,T}(p\,|\,\beta) - r ) \cdot P -c_f+ \tilde{V}_{t,T}(r_{t,T}(p\,|\,\beta),P\,|\,\beta) &\text{ if } P \in \mathscr P_\text{med}(r),\\
-c_f+\tilde{V}_{t,T}(r,P\,|\,\beta) &\text{ if } P \in \mathscr P_\text{high}(r).\\
\end{cases}
\end{equation*}
Differentiating in $r$ and adding $P$ to both sides, we have
\begin{equation*}
P + \partial_r V_{t,T}(r,p\,|\,\beta) = \begin{cases}
P+\partial_r \tilde{V}_{t,T}(r+\xmax,P\,|\,\beta) &\text{ if } P \in \mathscr P_\text{low}(r),\\
0 &\text{ if } P\in \mathscr P_\text{med}(r),\\
P+\partial_r \tilde{V}_{t,T}(r,p\,|\,\beta) &\text{ if } P \in \mathscr P_\text{high}(r).\\
\end{cases}
\end{equation*}
Recall the characterization of $r_{t,T}(p\,|\,\beta)$ minimizes the quantity $\tilde{r}\, p + \tilde{V}_{t,T}(\tilde{r},p\,|\,\beta)$. For $P \in \mathscr P_\text{low}(r)$, we can infer by convexity and the property $r+\xmax < r_{t,T}(p\,|\,\beta)$ that $P + \partial_r \tilde{V}_{t,T}(r+\xmax,P\,|\,\beta) \le 0$. Similar reasoning yields $P+\partial_r \tilde{V}_{t,T}(r,p\,|\,\beta) \ge 0$ for $P \in \mathscr P_\text{high}(r)$. Hence, $P + \partial_r V_{t,T}(r,p\,|\,\beta)$ is nondecreasing in $P$.

\vspace{7pt} \noindent  \textbf{Inductive Step (Part 2).}
Given this, we are now able to more easily analyze the main quantity of interest, $p+\partial_r \tilde{V}_{T-1,T}(r,p\,|\,\beta)$. By Lemma \ref{lem:sensitivity}, we have
 \begin{equation}
 \begin{aligned}
p+\partial_r \tilde{V}_{T-1,T}(r,p\,|\,\beta) &= (1-\lambda_t) \, \bigl[p+\mathbf{E}\bigl[ \partial_r V_{T,T}(r,p \, e^{-\kappa_Y} + \psi_{T}\,|\,\beta) \bigr] \bigr] \\
&+ \lambda_t \, \bigl[ p + \mathbf{E}\bigl[ \partial_r V_{T,T}(r,p \, e^{-\kappa_Y} + \psi_{T}\,|\,\beta) \, | \, \psi_{t} \ge \textnormal{VaR}_{\alpha_t}(\psi_{t})   \bigr] \bigl].
\end{aligned}
\label{eq:inductive1}
\end{equation}
The first term in brackets can be written as
\begin{align*}
p+\mathbf{E}\bigl[ \partial_r &V_{T,T}(r,p \, e^{-\kappa_Y} + \psi_{T}\,|\,\beta) \bigr] = p+ \mathbf{E}\bigl[ P + \partial_r V_{t,T}(r,p\,|\,\beta) - P \bigr]\\
&= (p -p \, e^{-\kappa_Y}) - \E(\psi_{t}) +\mathbf{E}\bigl[ P + \partial_r V_{t,T}(r,p\,|\,\beta) \bigr],
\end{align*}
which is clearly nondecreasing in $p$ given the property we proved in Part 1. The second term is the same, except with conditional expectations that do not depend on $p$.
\end{proof}

\marginalval*
\begin{proof}
We proceed via induction. The base case, that $\partial_r \tilde{V}_{T-1,T}(r,p\,|\,\beta)$, is nonincreasing in the spot price can be easily verified by essentially repeating the proof for the base case of Theorem \ref{thm:threshspot}. For the inductive step, we assume the property holds at $t$ and verify it for $t-1$. Let $P = p \, e^{-\kappa_Y} + \psi_{t}$. Following the definitions from the proof of Theorem \ref{thm:threshspot}, we have
\begin{equation*}
\partial_r V_{t,T}(r,p\,|\,\beta) = \begin{cases}
\partial_r \tilde{V}_{t,T}(r+\xmax,P\,|\,\beta) &\text{ if } P \in \mathscr P_\text{low}(r),\\
-P &\text{ if } P\in \mathscr P_\text{med}(r),\\
\partial_r \tilde{V}_{t,T}(r,p\,|\,\beta) &\text{ if } P \in \mathscr P_\text{high}(r).
\end{cases}
\end{equation*}
Although this function is clearly piecewise nonincreasing from the induction hypothesis, we would like to verify that $\partial_r V_{t,T}(r,p\,|\,\beta)$ nonincreasing over all $P$. We first consider the term $\partial_r \tilde{V}_{t,T}(r+\xmax,P\,|\,\beta)$. When $P \in \mathscr P_\text{low}(r)$ and the optimal unconstrained charge level $r_{t,T}(p\,|\,\beta)$ is not reachable from $r$, we can apply the same reasoning as in the proof of Theorem \ref{thm:threshspot} to deduce via convexity that $P+\partial_r \tilde{V}_{t,T}(r+\xmax,P\,|\,\beta) \le 0$. When $P \in \mathscr P_\text{med}(r)$, we know that $r_{t,T}(p\,|\,\beta)$ is reachable from $r$ and thus it follows that $P+\partial_r \tilde{V}_{t,T}(r+\xmax,P\,|\,\beta) \ge 0$. Rearranging these inequalities, we can see that
\begin{equation}
\partial_r V_{t,T}(r,p\,|\,\beta) = \min \bigl\{-P, \, \partial_r \tilde{V}_{t,T}(r+\xmax,P\,|\,\beta) \bigr\} \quad \text{if } P \in \mathscr P_\text{low}(r) \cup \mathscr P_\text{med}(r).
\label{eq:lowmed}
\end{equation}
By the induction hypothesis, $\partial_r V_{t,T}(r,p\,|\,\beta)$ is the minimum of two nonincreasing functions on $P \in \mathscr P_\text{low}(r) \cup \mathscr P_\text{med}(r)$.

Now let us consider the term $\partial_r \tilde{V}_{t,T}(r,p\,|\,\beta)$. For $P \in \mathscr P_\text{med}(r)$, the optimal charge level is reachable and thus $r \le r_{t,T}(p\,|\,\beta)$; by convexity $P+\partial_r \tilde{V}_{t,T}(r,p\,|\,\beta) \le 0$. For $P \in \mathscr P_\text{high}(r)$, we have $r > r_{t,T}(p\,|\,\beta)$ and thus $P+\partial_r \tilde{V}_{t,T}(r,p\,|\,\beta) \ge 0$. Therefore,
\begin{equation}
\partial_r V_{t,T}(r,p\,|\,\beta) = \max \bigl\{-P, \, \partial_r \tilde{V}_{t,T}(r,p\,|\,\beta) \bigr\} \quad \text{if } P \in \mathscr P_\text{med}(r) \cup \mathscr P_\text{high}(r).
\label{eq:medhigh}
\end{equation}
Combining (\ref{eq:lowmed}) and (\ref{eq:medhigh}), we see that $\partial_r V_{t,T}(r,p\,|\,\beta)$ is nonincreasing over all $P \in \mathbb R$. This proves the first part of the lemma. To conclude the second part, we invoke Lemma \ref{lem:sensitivity} to get
 \begin{equation*}
 \begin{aligned}
\partial_r \tilde{V}_{T-1,T}(r,p\,|\,\beta) = (1 -\lambda_t) \, &\mathbf{E}\bigl[ \partial_r V_{T,T}(r,p \, e^{-\kappa_Y} + \psi_{T}\,|\,\beta) \bigr] \\
&+ \lambda_t \, \mathbf{E}\bigl[ \partial_r V_{T,T}(r,p \, e^{-\kappa_Y} + \psi_{T}\,|\,\beta) \, | \, \psi_{t} \ge \textnormal{VaR}_{\alpha_t}(\psi_{t})   \bigr].
\end{aligned}
\end{equation*}
The right-hand-side is clearly nonincreasing in $p$.
\end{proof}

\thmcompatone*
\begin{proof}
By Lemma \ref{lem:sensitivity}, we have the equation
 \begin{equation}
 \begin{aligned}
\partial_r \tilde{V}_{0,T}(r,p\,|\,\beta) = (1 -\lambda_0) \, &\mathbf{E}\bigl[ \partial_r V_{1,T}(r,p \, e^{-\kappa_Y} + \psi_{1}\,|\,\beta) \bigr] \\
&+ \lambda_0 \, \mathbf{E}\bigl[ \partial_r V_{1,T}(r,p \, e^{-\kappa_Y} + \psi_{1}\,|\,\beta) \, | \, \psi_{1} \ge \textnormal{VaR}_{\alpha_0}(\psi_{1})   \bigr].
\label{eq:varless}
\end{aligned}
\end{equation}
We note that by the dynamic programming property of Theorem \ref{thm:bellman}, the value function $V_{1,T}^\beta$ at time $t=1$ depends only on $\beta_{-0}^*$ and not on the risk-aversion of the first time period $\beta_0 = (\lambda_0, \alpha_0)$. By the nonincreasing property of $\partial_r V_{1,T}^\beta(r,p)$ in $p$ from Lemma \ref{lem:marginalval}, it follows that 
\begin{equation*}
\mathbf{E}\bigl[ \partial_r V_{1,T}(r,p \, e^{-\kappa_Y} + \psi_{1}\,|\,\beta) \, | \, \psi_{1} \ge \textnormal{VaR}_{\alpha_0}(\psi_{1})   \bigr] \le \mathbf{E}\bigl[ \partial_r V_{1,T}(r,p \, e^{-\kappa_Y} + \psi_{1}\,|\,\beta) \bigr]
\end{equation*}
and $\alpha_0 \mapsto \mathbf{E}\bigl[ \partial_r V_{1,T}(r,p \, e^{-\kappa_Y} + \psi_{1}\,|\,\beta) \, | \, \psi_{1} \ge \textnormal{VaR}_{\alpha_0}(\psi_{1})   \bigr]$ is nonincreasing. Combining these two properties allows us to conclude from (\ref{eq:varless}) that $\partial_r \tilde{V}_{0,T}(r,p\,|\,\beta)$ is nonincreasing in $\beta$ over $\mathcal B(\beta^*_{-0})$, which in turn implies that the charging threshold $r_{0,T}(p\,|\,\beta)$ is \emph{nondecreasing}. The thresholds for $t > 0$ remain unchanged when $\beta$ (again by the dynamic programming property) and thus, it is easy to conclude that $\mathbf{R}_T^\beta \le \mathbf{R}_T^{\beta'}$ almost surely when $\beta \le \beta'$. Risk compatibility follows by the resource-monotonicity property of $f_T^-$. 
\end{proof}

\clearpage
\bibliographystyle{plainnat}
\bibliography{/Users/drjiang/Documents/Dropbox/Pittsburgh/Bibtex/Bib,/Users/drjiang/Documents/Dropbox/Pittsburgh/Bibtex/Risk,/Users/drjiang/Documents/Dropbox/Pittsburgh/Bibtex/EV}

\end{document}